\documentclass[onefignum,onetabnum]{siamart251216}

\usepackage{lipsum}
\usepackage{amsfonts}
\usepackage{graphicx}
\usepackage{epstopdf}
\usepackage{algorithmic}
\ifpdf
  \DeclareGraphicsExtensions{.eps,.pdf,.png,.jpg}
\else
  \DeclareGraphicsExtensions{.eps}
\fi

\newsiamremark{remark}{Remark}
\newsiamremark{assumption}{Assumption}
\crefname{assumption}{Assumption}{Assumptions}
\newsiamremark{hypothesis}{Hypothesis}
\crefname{hypothesis}{Hypothesis}{Hypotheses}
\newsiamthm{claim}{Claim}
\newsiamremark{fact}{Fact}
\crefname{fact}{Fact}{Facts}

\headers{Multilevel Gaussian Processes for Inverse Problems}{S. Salatovic, A. L. Teckentrup, and S. Krumscheid}

\title{Multilevel Gaussian Processes with Applications to Inverse Problems}

\author{Stjepan Salatovic\thanks{Karlsruhe Institute of Technology, Scientific Computing Center, 76131 Karlsruhe, Germany
    (\email{stjepan.salatovic@kit.edu}, \email{sebastian.krumscheid@kit.edu}).}
  \and Aretha L. Teckentrup\thanks{University of Edinburgh, School of Mathematics and Maxwell Institute, Edinburgh EH9 3FD, UK (\email{a.teckentrup@ed.ac.uk}).}
    \and Sebastian Krumscheid\footnotemark[1]}

\usepackage{amsopn}

\ifpdf
\hypersetup{
pdftitle={Multilevel Gaussian Processes with Applications to Inverse Problems},
pdfauthor={S. Salatovic, A. L. Teckentrup, and S. Krumscheid}
}
\fi

\usepackage{mathtools}
\usepackage{amssymb}
\usepackage{enumitem}

\newcommand{\R}{\mathbb{R}}

\newcommand{\N}{\mathbb{N}}
\newcommand{\E}{\mathbb{E}}

\newcommand{\pos}[1]{\left(#1\right)_+}

\DeclareMathOperator*{\argmax}{arg\,max}

\begin{document}

\maketitle

\begin{abstract}
We present a multilevel Gaussian process (GP) regression framework for approximating functions that are accessible only through a hierarchy of approximations of increasing accuracy and computational cost.
One example is the response surface of a parametrized partial differential equation, where the hierarchy arises from the resolution of the numerical solver.
Modeling level increments independently with Matérn-kernel GPs and allocating samples flexibly across discretization levels yields substantial computational savings over single-level GP regression.
We derive $L^2$ convergence rates for the posterior mean and variance, accounting for maximum likelihood estimation of the marginal variance and for kernel misspecification.
When increment regularity or solver cost is unknown a priori, an adaptive sequential-design algorithm with cost-adapted active learning criteria allocates samples across levels.
We apply the surrogate to Bayesian inverse problems, for which Hellinger distance bounds between the true and approximate posteriors are known.
Numerical experiments illustrate the theoretical rates and the resulting computational savings.
\end{abstract}

\begin{keywords}
Gaussian process regression, multilevel Monte Carlo, surrogate modeling,
Bayesian inverse problems, sequential experimental design
\end{keywords}

\begin{MSCcodes}
60G15, 62F15, 62K05, 65D15, 41A25
\end{MSCcodes}


\section{Introduction}
\label{sec:intro}

Many problems in uncertainty quantification and scientific computing require the approximation of a function
\[
f : \Omega \subset \R^d \to \R,
\]
whose evaluation is computationally expensive.
Such functions arise, for instance, when describing uncertainty
propagation in terms of statistical moments, in optimization procedures, or in
the solution of Bayesian inverse problems.
In many applications, $f$ cannot be evaluated exactly and is instead accessible only through a hierarchy of approximations of increasing accuracy and
computational cost.
More generally, such a hierarchy may arise from varying the resolution of a
numerical solver, the tolerance of an iterative method, or the fidelity of a
physical model.
A prototypical example, which we use as a running example throughout this
work, is the case where $f$ is a so-called response surface derived from the solution
of a parametrized partial differential equation (PDE),
\[
\omega \mapsto
f(\omega) \coloneqq Q(u(\cdot,\omega)),
\]
where $u(\cdot,\omega)$ denotes the solution of the PDE for a parameter realization $\omega \in \Omega$ and $Q$ is a quantity of interest.
Here, the underlying PDE typically does not admit an analytical solution, so
evaluations of $f$ rely on numerical approximations, which are only approximate
and, for complex models, computationally demanding, since each evaluation
requires solving the governing PDE for a given parameter realization
$\omega \in \Omega$.
This motivates the construction of surrogate models that accurately approximate $f$ while requiring only a limited number of function evaluations \cite{Sacks1989, Kennedy2001}.
Once constructed, the surrogate can be used efficiently in downstream computational tasks such as Monte Carlo sampling.
Examples of deterministic surrogate models include polynomial approximation \cite{Cohen2015}, neural networks, splines, and reduced-order models \cite{Hesthaven2016}.
A further example is stochastic collocation \cite{Babuska2010} which, despite its name, is likewise deterministic, constructing a structured interpolant in the stochastic parameter space from evaluations of the model at carefully chosen collocation points.
Probabilistic methods, such as Bayesian neural networks or Gaussian processes \cite{rasmussenGaussianProcessesMachine2006a}, additionally quantify predictive uncertainty, thereby providing a built-in error estimate.

In this work, we focus on approximating $f$ using Gaussian process (GP) regression.
For sufficiently smooth functions, this method achieves optimal convergence rates, provided the training data and model hyperparameters are chosen well \cite{teckentrupConvergenceGaussianProcess2020}.
Optimality here is in the sense that the resulting error rate can not be improved by any algorithm based on the same number of point evaluations.
Attaining a prescribed accuracy, however, requires a sufficiently large number of training samples.
Since each evaluation of $f$ involves solving the underlying PDE, the total cost is dominated by these expensive evaluations and quickly becomes prohibitive.
To mitigate these costs, we use multiple discretizations of a numerical solver, which provide better accuracy as the computational effort increases.
For a finite element method, for instance, successive levels correspond to progressively finer meshes.
Multilevel Monte Carlo methods \cite{gilesMultilevelMonteCarlo2015, Heinrich2001, Cliffe2011} exploit such a hierarchy by combining information across levels, computing most of the work on cheap coarse discretizations while using only a few expensive fine-level samples, which is shown to substantially reduce the overall computational complexity.
Early works targeted the computation of expected values \cite{Giles2008}, while others estimated additional statistical quantities \cite{krumscheidQuantifyingUncertainSystem2020, ayoul-guilmardQUANTIFYINGUNCERTAINSYSTEM2023, Giles2015}, or the approximation of functions themselves \cite{krumscheidMultilevelMonteCarlo2018}.
Multilevel methods have since then also been applied to stochastic collocation \cite{teckentrupMultilevelStochasticCollocation2015}, kernel-based interpolation \cite{nobileSparseApproximationMultilinear2018, sungStackingDesignsDesigning2024}, and optimal sampling for polynomial regression \cite{Haji-Ali2020}.

Similar to these works, we introduce and analyze a novel multilevel GP framework for the approximation of a function.
Although there are some approaches for studying hierarchical constructions of GPs \cite{Kennedy2000, gratietRECURSIVECOKRIGINGMODEL2014,nobileSparseApproximationMultilinear2018, sungStackingDesignsDesigning2024, eharaADAPTIVESTRATEGYSEQUENTIAL2023}, a systematic multilevel extension to GP regression that includes the surrogate’s variance has received no attention.
A key distinction from the autoregressive multi-fidelity approaches of \cite{Kennedy2000, gratietRECURSIVECOKRIGINGMODEL2014} is that we model the level differences by independent GPs, whereas those works condition each higher-fidelity model on the lower-fidelity ones.
While the autoregressive construction could potentially improve approximation quality by exploiting cross-level correlations, the independent formulation adopted here leads to a considerably simpler implementation and a more tractable theoretical analysis.
The adaptive multilevel approach of \cite{eharaADAPTIVESTRATEGYSEQUENTIAL2023} is also related, which models the unknown function as a GP itself and adaptively allocates samples for a prescribed computational budget.
We treat the function to be deterministic, focus primarily on a rigorous error and complexity analysis, and aim to attain a specified error tolerance.
Finally, due to the equivalence of a GP’s posterior mean with a kernel interpolant \cite{kanagawaGaussianProcessesKernel2018}, the kernel-based multilevel approaches of \cite{nobileSparseApproximationMultilinear2018, sungStackingDesignsDesigning2024} are comparable to the deterministic part of the surrogate presented here.
However, the present work explicitly exploits the stochastic component of the GP, which is missing in the kernel-based approaches, and therefore additionally enables uncertainty quantification.
Controlling the resulting predictive variance requires an appropriate choice of level-dependent hyperparameters, most notably the marginal variance of each prior kernel.
For this, we build on the convergence analysis of the maximum likelihood estimator in \cite{karvonenMaximumLikelihoodEstimation2020}, accounting for the effect of smoothness misspecification, which is not considered in the prior works mentioned above.
As an application, we consider Bayesian inverse problems \cite{stuartInverseProblemsBayesian2010}, incorporating bounds on the Hellinger distance between the true and surrogate posterior \cite{stuartPosteriorConsistencyGaussian2017}, and identify regimes in which the multilevel approach yields substantial computational savings over its single-level counterpart.
Finally, an adaptive algorithm for sample allocation across levels is proposed for settings where the problem-dependent convergence rates are not known a priori.
The theoretical findings are illustrated by numerical experiments.

The paper is organized as follows. We present the basics and convergence properties of GPs in Section \ref{sec:gpr}.
In Section \ref{sec:mlgp} we introduce the multilevel GP and carry out the corresponding error analysis.
Section \ref{sec:bip} applies our results to Bayesian inverse problems.
In Section \ref{sec:adaptive}, we present an adaptive algorithm for the construction of multilevel GPs.
Section \ref{sec:numerics} illustrates the theoretical rates in numerical experiments, and a conclusion is given in Section \ref{sec:conclusion}.


\section{Gaussian process regression}
\label{sec:gpr}

In the following section, we give a short background on GP conditioning, a specific kernel choice, and approximation quality.
For this, let $f : \Omega \to \R$ be an unknown deterministic function defined on $\Omega \subset \R^d$.
Throughout this work, we assume that $\Omega \subset \R^d$ is a bounded domain with Lipschitz boundary, satisfying an interior cone condition \cite[Section 3.3]{Wendland2004}.
The goal of GP regression is to construct an approximation of $f$ from finitely many point evaluations.
This is achieved by placing a probabilistic prior on $f$ in the form of a GP, $f_\mathrm{GP} \sim \mathrm{GP}(m, k)$, specified by a mean function $m : \Omega \to \R$ and a covariance (or kernel) function $k : \Omega \times \Omega \to \R$.
This means that for any finite set of points $\{\omega_1, \ldots, \omega_N\} \subset \Omega$ of any size $N \in \N$, the random vector
\[
\left(f_\mathrm{GP}(\omega_1), \ldots, f_\mathrm{GP}(\omega_N)\right) \in \R^N
\]
is multivariate normally distributed with mean vector $(m(\omega_i))_{i=1}^N$ and covariance matrix $(k(\omega_i, \omega_j))_{i,j=1}^N$.
Denoting such a set of design points by $D_N := \{\omega_1, \ldots, \omega_N\}$ and conditioning the prior process $f_\mathrm{GP}$ on corresponding function evaluations
\[
f(D_N) := \{f(\omega_1), \ldots, f(\omega_N)\}
\]
yields again a GP \cite{rasmussenGaussianProcessesMachine2006a}, namely the posterior
$$
f_\mathrm{GP} \mid f(D_N) \sim \mathrm{GP}(m_N^f, k_N),
$$
with posterior mean $m_N^f: \Omega \to \R$ and covariance kernel $k_N: \Omega \times \Omega \to \R$ given by
\begin{align}
    m_N^f(\omega)
    &= m(\omega) + k(\omega, D_N)^\top k(D_N, D_N)^{-1} \bigl(f(D_N) - m(D_N)\bigr),
    \label{eq:posterior_mean} \\
    k_N(\omega,\omega')
    &= k(\omega,\omega') - k(\omega, D_N)^\top k(D_N, D_N)^{-1} k(D_N, \omega').
    \label{eq:posterior_cov}
\end{align}
Here, $k(\omega, D_N) \in \R^N$ denotes the vector with entries $(k(\omega,\omega_i))_{i=1}^N$ and $k(D_N, D_N) \in \R^{N \times N}$ is the kernel matrix with entries $(k(\omega_i,\omega_j))_{i,j=1}^N$.
The posterior mean $m_N^f$ often serves as the surrogate for $f$, while the posterior covariance $k_N$ quantifies the remaining model uncertainty.
In what follows, we are often interested in the integrated posterior variance
\[
\|k_N^{1/2}\|_{L^2(\Omega)}^2 = \int_\Omega k_N(\omega,\omega)\, d\omega,
\]
for which $k_N^{1/2}$ is interpreted as the mapping $\omega \mapsto k_N(\omega,\omega)^{1/2}$ (also known as the \emph{power function} in the kernel literature).
In the context of Bayesian inverse problems, as we will see in Section \ref{sec:bip}, one can incorporate all available surrogate information to build approximations, including the uncertainty arising from a finite dataset.
Notably, the posterior covariance $k_N$ depends only on the prior kernel $k$ and the design $D_N$, but not on the function values $f(D_N)$, which shows that the uncertainty in a GP is purely model-based, reflecting the geometry of the design $D_N$ rather than the observed data.

Throughout this work, we consider the noiseless observation setting, where the data are exact evaluations $f(\omega_i)$ with no additive observation noise.
This results in the posterior mean \eqref{eq:posterior_mean} interpolating the data, $m_N^f(\omega_i) = f(\omega_i)$, and the posterior variance vanishing, $k_N(\omega_i, \omega_i) = 0$, so that the GP regression is actually an interpolation.
This is appropriate in the present setting, where $f$ is a deterministic quantity of interest evaluated by a numerical solver, so repeated evaluations at the same parameter are reproducible and have no statistical error.
Obviously, there will be a model error depending on the discretization, which, however, depends deterministically on the parameter and has no random perturbations.
This model error is explicitly accounted for later in Section \ref{sec:error_analysis} and is not treated as an observational noise term.


\subsection{Matérn kernels and reproducing kernel Hilbert spaces}

We briefly recall the notion of a reproducing kernel Hilbert space \cite{Berlinet2004} and its characterization for the Matérn kernel.
Every \emph{positive-definite} kernel $k: \Omega \times \Omega \rightarrow \mathbb{R}$, which means that the matrix $k(D_N, D_N)$ is positive-definite for any design $D_N \subset \Omega$ of any $N \in \N$ distinct points, induces a unique reproducing kernel Hilbert space (RKHS), denoted by $\mathcal{H}_{k}(\Omega)$.
This RKHS is equipped with inner product $\langle \cdot, \cdot \rangle_{\mathcal{H}_{k}(\Omega)}$, inducing the norm $\|\cdot\|_{\mathcal{H}_{k}(\Omega)}$, and it has the properties that
\begin{equation}
\label{eq:reproducing_prop}
    k(\cdot, \omega) \in \mathcal{H}_{k}(\Omega) \quad\text{and}\quad f(\omega) = \langle f, k(\cdot, \omega) \rangle_{\mathcal{H}_{k}(\Omega)}
\end{equation}
for all $f \in \mathcal{H}_{k}(\Omega)$ and all $\omega \in \Omega$.
The latter is referred to as the reproducing property.
Throughout this work we use the Matérn covariance kernel
\begin{equation}
\label{eq:matern_kernel}
    k(\omega, \omega') = \frac{2^{1 - \nu}}{\Gamma(\nu)}\left(\frac{\|\omega - \omega'\|_2}{\lambda}\right)^\nu B_\nu\!\left( \frac{\|\omega -\omega'\|_2}{\lambda} \right),
\end{equation}
where $\Gamma$ denotes the Gamma function, $B_\nu$ the modified Bessel function of the second kind of order $\nu$, and $\|\cdot\|_2$ the Euclidean norm.
The hyperparameters are the length scale $\lambda > 0$ and the smoothness parameter $\nu > 0$.
The length scale $\lambda$ controls the decay of correlations of input points with distance, and thus the typical length over which the process varies, while the smoothness $\nu$ determines the differentiability of the resulting sample paths.
Note that \eqref{eq:matern_kernel} is defined in a normalized way with unit marginal variance $k(\omega,\omega) = 1$.
We will introduce a pointwise marginal variance $\sigma^2 > 0$ below by rescaling the kernel to $\sigma^2 k$.

The RKHS of a Matérn kernel can be described in terms of Sobolev spaces.
For $\tau \in \mathbb{N}$, the Sobolev space $W_2^\tau(\Omega)$ consists of all functions in $L^2(\Omega)$ whose weak derivatives up to order $\tau$ are square integrable,
\[
W_2^\tau(\Omega) :=
\bigl\{f \in L^2(\Omega) : D^\alpha f \in L^2(\Omega)
\;\forall\, \alpha \in \mathbb{N}^d,\; |\alpha| \leq \tau \bigr\}.
\]
For non-integer $\tau > d/2$, the fractional Sobolev space on $\mathbb{R}^d$ is defined via the Fourier transform $\hat{f}$ of $f$,
\[
W_2^\tau(\mathbb{R}^d) :=
\bigl\{f \in L^2(\mathbb{R}^d) :
\int_{\mathbb{R}^d} (1 + \|\xi\|_2^2)^\tau \,|\hat{f}(\xi)|^2 \,d\xi < \infty \bigr\}.
\]
On a bounded domain $\Omega \subset \R^d$, the fractional Sobolev space is then obtained by restriction,
\[
W_2^\tau(\Omega) := \bigl\{ f|_\Omega : f \in W_2^\tau(\mathbb{R}^d) \bigr\},
\]
equipped with the norm $\|g\|_{W_2^\tau(\Omega)} := \inf\{ \|f\|_{W_2^\tau(\mathbb{R}^d)} : f|_\Omega = g \}$.
If $k$ is a Matérn kernel with smoothness parameter $\nu > 0$, then its RKHS $\mathcal{H}_k(\Omega)$, as a set, coincides with $W_2^\tau(\Omega)$ for $\tau = \nu + d/2$.
Sample paths of $\mathrm{GP}(0,k)$, on the other hand, are rougher and lie a.s. in $W_2^{s}(\Omega)$ for every $s < \nu$.
In other words, $\mathrm{GP}(0,k)$ samples are thus less regular by $d/2$ derivatives compared to functions from the RKHS.
In particular, these samples are a.s. not in $\mathcal{H}_k(\Omega)$
\cite{kanagawaGaussianProcessesKernel2018}.
In addition to the set equivalence, the corresponding norms are equivalent, which means there exist constants $c_1, c_2 > 0$ independent of $f$, such that
\[
c_1 \|f\|_{W_2^\tau(\Omega)}
\leq \|f\|_{\mathcal{H}_k(\Omega)}
\leq c_2 \|f\|_{W_2^\tau(\Omega)}
\qquad \forall\, f \in W_2^\tau(\Omega).
\]
We write $\mathcal{H}_k(\Omega) \simeq W_2^\tau(\Omega)$ in this situation, and use $\,\simeq\,$ to denote these equivalences of Hilbert spaces.
Throughout, we also write $a \lesssim b$ if there exists a constant $C > 0$ such that $a \leq C b$, and $a \asymp b$ if both $a \lesssim b$ and $b \lesssim a$.
The constant may depend on the problem, for example on the domain $\Omega$, its dimension $d$, the fixed kernel hyperparameters $\lambda$ and $\nu$, or the
function $f$, but not on the design $D_N$ or its size $N$.


\subsection{Convergence of GP regression}

The accuracy of GP regression is mostly controlled by the regularity of the target function $f$, kernel choice $k$, and spatial distribution of the design points in $D_N$.
The latter is measured using the so-called fill distance and the separation radius.
For a design $D_N \subset \Omega$, these are defined as
\[
h_{D_N} :=  \sup_{x \in \Omega} \inf_{\omega \in D_N} \|x - \omega\|_2,
\qquad
q_{D_N} := \frac{1}{2}\min_{\substack{\omega \neq \omega' \in D_N}}
\|\omega - \omega'\|_2,
\]
respectively.
Intuitively, the fill distance $h_{D_N}$ is the radius of the largest ball in $\Omega$ containing no design point, and thus measures how well $D_N$ covers the domain, while the separation radius $q_{D_N}$ is half the smallest distance between distinct design points.
A small separation radius $q_{D_N}$ indicates nearly identical points, which leads to ill-conditioned kernel matrices.
The ratio  $\rho_{D_N} := h_{D_N} / q_{D_N}$ is called the mesh ratio.
A sequence of designs is called \emph{quasi-uniform} if the mesh ratio is uniformly bounded in $N$, that is $\rho_{D_N} \lesssim 1$.
Quasi-uniform designs attain the optimal fill-distance and separation radius scaling $h_{D_N} \asymp q_{D_N} \asymp N^{-1/d}$ \cite[Proposition~14.1]{Wendland2004}.

Besides the design, the kernel carries hyperparameters that must be specified.
Throughout this work, we fix the length scale $\lambda$ and smoothness $\nu$ a priori, and introduce a marginal variance $\sigma^2 > 0$ by scaling the unit-variance Matérn kernel, $\sigma^2 k$.
The marginal variance, also referred to as the \emph{scale}, controls the overall magnitude of the process and thereby the width of the predictive uncertainty bands.
We use both terms, marginal variance and scale, interchangeably.
The scale can be estimated from the data by maximum likelihood estimation (MLE) \cite{rasmussenGaussianProcessesMachine2006a}, which in the interpolation setting with no observational noise admits the closed form
\begin{equation}
\label{eq:mle}
\hat\sigma^2_N = \frac{1}{N}\, f(D_N)^\top k(D_N, D_N)^{-1} f(D_N),
\end{equation}
where $f(D_N)$ are the observations and $k(D_N, D_N)$ the corresponding kernel matrix.
From the posterior covariance \eqref{eq:posterior_cov}, replacing the prior kernel $k$ by $\hat\sigma_N^2\, k$ scales the posterior kernel $k_N$ by $\hat\sigma_N^2$, so the estimated scale enters the posterior variance linearly and thus directly affects the uncertainty quantified by the GP.
The following theorem bounds it together with the posterior mean and variance.

\begin{theorem}[Convergence rates for GP regression]
\label{thm:gp_error}
Let $k$ be a Matérn kernel with $\mathcal{H}_k(\Omega) \simeq W_2^{\tau}(\Omega)$ for $\tau = \nu + d/2$, and let
$f \in W_2^{\tau_f}(\Omega)$ with $\tau_f > d/2$.
Then there exists $h_0 > 0$, independent of $f$ and $D_N$, such that
for every design with $h_{D_N} \leq h_0$ the following hold, with constants
$C > 0$ independent of $f$ and $D_N$:

\medskip
\noindent\emph{(i) Posterior mean:}
\[
\|f - m_N^f\|_{L^2(\Omega)}
\;\leq\; C\, h_{D_N}^{\min(\tau,\,\tau_f)}\,
\rho_{D_N}^{\pos{\tau - \tau_f}}\, \|f\|_{W_2^{\tau_f}(\Omega)}
\]

\noindent\emph{(ii) Posterior variance:}
\[
\|k_N^{1/2}\|_{L^2(\Omega)}
\;\leq\; C\, h_{D_N}^{\tau - d/2}
\]

\noindent\emph{(iii) Marginal-variance MLE:}
\[
\hat\sigma_N
\;\leq\; C\, N^{-1/2}\, q_{D_N}^{-\pos{\tau - \tau_f}}\,
\|f\|_{W_2^{\tau_f}(\Omega)}
\]
\end{theorem}

\begin{proof}
All three bounds specialize known results to the present setting, in which the hyperparameters $\lambda$ and $\nu$ are fixed and the prior mean is zero.

\emph{(i)} This is \cite[Theorem~3.5]{teckentrupConvergenceGaussianProcess2020}, where the choice $\beta = 0$ gives the $L^2$ norm and fixing the hyperparameters removes the dependence on the estimated smoothness.

\emph{(ii)} The posterior covariance $k_N$ in \eqref{eq:posterior_cov} depends only on the kernel $k$ and the design $D_N$, not on $f$, so its rate is governed by the kernel smoothness $\tau$ solely.
Since $\Omega$ is bounded, we have
\[
\|k_N^{1/2}\|_{L^2(\Omega)} \;\leq\; |\Omega|^{1/2}\, \|k_N^{1/2}\|_{L^\infty(\Omega)}.
\]
The uniform bound $\|k_N^{1/2}\|_{L^\infty(\Omega)} \leq C\, h_{D_N}^{\tau - d/2}$ follows from the worst-case representation \cite[Proposition 3.5]{stuartPosteriorConsistencyGaussian2017}
\[
k_N(\omega,\omega)^{1/2} = \sup_{\|g\|_{\mathcal{H}_k(\Omega)} \leq 1} |g(\omega) - m_N^g(\omega)|,
\]
together with the bound in \cite[Theorem 1]{JMLR:v22:20-662} using the $L^\infty(\Omega)$ norm.

\emph{(iii)} The estimated scale is bounded in \cite[Proposition~3.1]{karvonenMaximumLikelihoodEstimation2020} for the well-specified case $\tau_f \geq \tau$, and in \cite[Proposition~4.5]{karvonenMaximumLikelihoodEstimation2020} for the misspecified case $\tau_f < \tau$.
\end{proof}

The three bounds describe different aspects of the same interpolation problem.
The posterior mean error~(i) and posterior variance~(ii) decrease as the design is refined, i.e., $N$ is increased, at rates governed by the kernel smoothness $\tau$ and the target regularity $\tau_f$.
Note that bound~(ii) is stated for the posterior variance of a \emph{unit-variance} kernel, such as the Matérn kernel defined in \eqref{eq:matern_kernel}.
For a kernel with marginal variance $\sigma^2$, the bound scales accordingly, with the constant $C$ depending linearly on $\sigma$ \cite[Remark~3.9]{teckentrupConvergenceGaussianProcess2020}.
In our setting, the scale is estimated, and its impact is captured separately by bound~(iii), which controls its magnitude.

When $\tau_f \geq \tau$, the target function $f$ is at least as smooth as functions from the kernel’s RKHS, which we refer to as the \emph{well-specified} regime.
When the kernel is smoother than the function it approximates, i.e., $\tau > \tau_f$, known as the so-called \emph{misspecified} regime, the exponent $\pos{\tau - \tau_f}$ appearing in bounds~(i) and (iii) quantifies the amount of misspecification.
For bound~(i), this penalty enters through the mesh ratio $\rho_{D_N}$, but can be escaped by using quasi-uniform designs, for which the mesh ratio is bounded independent of $N$.
For the MLE bound~(iii), however, it enters through the separation radius $q_{D_N}$, which behaves as $N^{-1/d}$ under quasi-uniformity, such that the bound reads
\begin{equation}
\label{eq:quasi_uniform_scale}
\hat\sigma_N \lesssim N^{-1/2 + (\tau-\tau_f)_+/d}.
\end{equation}
Thus, in the strongly misspecified regime of $\tau > \tau_f + d/2$, which means that $f$ is rougher than $\mathrm{GP}(0, k)$ samples \cite{Steinwart2019, kanagawaGaussianProcessesKernel2018}, the exponent of \eqref{eq:quasi_uniform_scale} becomes positive such that the estimated marginal variance does not need to decay with $N$.
Under additional regularity assumptions on $f$, this rate in bound~(iii) is sharp up to logarithmic factors \cite{karvonenMaximumLikelihoodEstimation2020}.

Since the marginal variances are not known but estimated in this work, the
relevant measure of posterior uncertainty is not solely the posterior variance but its product with the estimated scale, $\hat\sigma_N\, k_N^{1/2}$.
The following corollary combines bounds~(ii) and~(iii) into a single rate for this quantity, specialized to quasi-uniform designs.

\begin{corollary}[MLE-scaled posterior variance]
\label{cor:mle_scaled_var}
Let the assumptions of Theorem \ref{thm:gp_error} hold, and let $D_N$ be quasi-uniform.
Then there exists $C > 0$ independent of $f$ and $N$ such that
\[
\|\hat\sigma_N\, k_N^{1/2}\|_{L^2(\Omega)}
\;\leq\; C\, N^{-\min(\tau, \tau_f)/d}\, \|f\|_{W_2^{\tau_f}(\Omega)}.
\]
\end{corollary}

\begin{proof}
We have $\|\hat\sigma_N \, k_N^{1/2}\|_{L^2(\Omega)}
= \hat\sigma_N\, \|k_N^{1/2}\|_{L^2(\Omega)}$ and we bound the two factors using Theorem \ref{thm:gp_error}.
For the posterior standard deviation, the variance bound from part~(ii) of Theorem \ref{thm:gp_error} together with the quasi-uniformity $h_{D_N} \asymp N^{-1/d}$ and $\rho_{D_N} \lesssim 1$ gives
\[
\|k_N^{1/2}\|_{L^2(\Omega)}
\;\lesssim\; N^{-\tau / d + 1/2}.
\]
For the estimated scale, the MLE bound from part~(iii) of the same theorem, together with $q_{D_N} \asymp N^{-1/d}$, yields
\[
\hat\sigma_N \lesssim\; N^{-1/2}\, q_{D_N}^{-\pos{\tau - \tau_f}} \,\|f\|_{W_2^{\tau_f}(\Omega)} \lesssim\; N^{-1/2 \,+\, \pos{\tau - \tau_f}/d}\,
\|f\|_{W_2^{\tau_f}(\Omega)}.
\]
Multiplying the bounds and noting that
\[
\pos{\tau - \tau_f} - \tau = - \min(\tau, \tau_f)
\]
yields the claim.
\end{proof}

Note that the rate $\min(\tau,\tau_f)/d$ coincides with the convergence rate of the posterior mean in bound~(i) of Theorem \ref{thm:gp_error}.
The scaled posterior standard deviation therefore decreases at the same rate as the approximation error itself.
In the well-specified case $\tau_f \geq \tau$ this rate is $\tau/d$, while for $\tau_f < \tau$ it reduces to $\tau_f/d$.


\section{Multilevel Gaussian Processes}
\label{sec:mlgp}

In practice, the exact response surface $f$ introduced in Section \ref{sec:intro} is not directly accessible, as its evaluation requires the exact solution of the underlying differential equation.
Instead, one has usually access to a sequence of approximate response surfaces
\[
f_n : \Omega \to \mathbb{R}, \qquad
f_n(\omega) := Q(u_n(\cdot,\omega)),
\]
where $u_n(\cdot,\omega)$ is the numerical solution obtained with discretization parameter $n \in \mathbb{N}$.
For instance, $n$ could denote the number of finite elements or the number of grid points in a finite difference scheme.
As the discretization parameter $n$ increases, $f_n$ converges to $f$ in an appropriate sense, but at the expense of a higher cost per evaluation.
Since an accurate GP approximation requires a large sample size $N$ of evaluations of a sufficiently finely discretized response $f_n$, this can quickly become a computational bottleneck.
To overcome this, we employ the multilevel approach, which combines many cheap and coarse evaluations with a small number of expensive and fine samples, and we show that both the posterior mean and variance can be controlled efficiently.

To this end, consider $L + 1 \in \mathbb{N}$ levels and an increasing sequence of discretization parameters $n_0 < n_1 < \cdots < n_L$.
Typically, we will use the sequence $n_\ell = \lceil C_n M^\ell \rceil$ for $\ell \in \N_0$ with a refinement factor $M \in \N_{>1}$ and a constant $C_n > 0$, where $\lceil x \rceil$ denotes the smallest integer not below $x$.
Writing $f_\ell := f_{n_\ell}$ for $\ell \in \{0, \ldots, L\}$, the most accurate response surface $f_L \approx f$ can be expressed as the
telescoping sum
\begin{equation}\label{eq:telescop}
f_L = f_0 + \sum_{\ell=1}^L f_\ell - f_{\ell-1}.
\end{equation}
The idea behind the telescoping identity is that even a coarse approximation $f_0$ captures the qualitative behavior of $f$, while the successive difference terms $f_\ell - f_{\ell-1}$ correct increasingly fine details.
Rather than approximating $f$ from the left-hand side of~\eqref{eq:telescop}, which would require a single fine discretization for the GP construction, we exploit the right-hand side and build a separate GP for each level.
To unify notation, we define $\Delta_0 := f_0$ and the level increments
\[
\Delta_\ell := f_\ell - f_{\ell-1}, \quad \ell \in \{1, \ldots, L\},
\]
such that $f_L = \sum_{\ell=0}^L \Delta_\ell$.
For each level $\ell \in \{0, \ldots, L\}$, we then assign an independent GP prior
\[
\Delta_{\ell}^{\mathrm{GP}} \sim \mathrm{GP}(0,\, k^\ell),
\]
where
\begin{equation}
\label{eq:unit_var_matern}
k^\ell(\omega, \omega') := \sigma_\ell^2\, k(\omega, \omega')
\end{equation}
is a unit-variance Matérn kernel $k$ \eqref{eq:matern_kernel} with length scale $\lambda$ and smoothness $\nu$, scaled with level-dependent marginal variance $\sigma_\ell^2 > 0$.
The zero-mean prior is natural in this setting, since the discretized responses converge, $f_\ell \to f$ as $\ell \to \infty$, so the increments $\Delta_\ell = f_\ell - f_{\ell-1}$ vanish accordingly.
Consecutive levels therefore carry no systematic offset, and $\Delta_\ell$ concentrates increasingly around zero.
What remains to be captured is the typical magnitude of these increments, which is encoded in the marginal variance $\sigma_\ell^2$.
Since the increments $\Delta_\ell$ decrease in a suitable norm with $\ell$, the corresponding marginal variances are expected to shrink across levels.
Note that this construction is different from the autoregressive multi-fidelity model of Kennedy and O'Hagan~\cite{Kennedy2000}, in which the fine-level GP is  obtained by conditioning on the coarse-level output, yielding correlated level models and a coupled posterior update.
The present formulation instead follows the classical multilevel Monte Carlo approach, where each level models the incremental correction $\Delta_\ell$ independently.

To update the prior GPs using data for each level, we consider level-dependent design sets, i.e., on level $\ell \in \{0, \ldots, L\}$ the design set
\[
D_\ell := \{\omega_{\ell,1}, \ldots, \omega_{\ell,N_\ell}\} \subset \Omega
\]
of size $N_\ell \in \N$ with corresponding response increment evaluations
$\Delta_\ell(D_\ell) \in \mathbb{R}^{N_\ell}$.
Conditioning the GP prior $\Delta_{\ell}^\mathrm{GP}$ on the data yields the level-$\ell$ posterior
\[
\Delta_\ell^\mathrm{GP} \mid \Delta_\ell(D_\ell)
\;\sim\; \mathrm{GP}\!\left(m_\ell,\, k_\ell\right),
\]
where, using the update formulas \eqref{eq:posterior_mean} and \eqref{eq:posterior_cov}, it follows that
\begin{align}
    m_\ell(\omega)
    &= k^\ell(\omega, D_\ell)^\top k^\ell(D_\ell, D_\ell)^{-1} \Delta_\ell(D_\ell),
    \label{eq:posterior_mean_level_l} \\
    k_\ell(\omega, \omega')
    &= k^\ell(\omega, \omega')
       - k^\ell(\omega, D_\ell)^\top k^\ell(D_\ell, D_\ell)^{-1} k^\ell(D_\ell, \omega').
    \label{eq:posterior_cov_level_l}
\end{align}
Each scale $\sigma_\ell^2$ in \eqref{eq:unit_var_matern} is estimated from the level data $\Delta_\ell(D_\ell)$ via MLE~\eqref{eq:mle}.
Note that the scale cancels out in \eqref{eq:posterior_mean_level_l}, so that the level posterior mean $m_\ell$ is independent of $\sigma_\ell$.
In particular, Theorem \ref{thm:gp_error} (i) also applies to $m_\ell$ with the estimated scale.
After conditioning each level separately, we can now put the levels together into one multilevel surrogate.
By the telescoping representation \eqref{eq:telescop}, the finest response is the sum of increments, $f_L = \sum_{\ell=0}^L \Delta_\ell$, which we approximate by the corresponding sum of level GPs, i.e., $f_L^{\mathrm{GP}} \coloneqq \sum_{\ell=0}^L \Delta_\ell^{\mathrm{GP}}$.
Since the level priors are independent, conditioning on all level data $\{\Delta_\ell(D_\ell)\}$ yields a multilevel posterior that is again a GP, namely
\begin{equation}
\label{eq:mlgp_posterior}
f_L^{\mathrm{GP}} \,\big|\, \{\Delta_\ell(D_\ell)\}_{\ell=0}^L 
\,\sim\, \mathrm{GP}\!\left(\bar{m}^f,\, \bar{k}\right).
\end{equation}
Its mean and covariance are obtained by summing those of the level posteriors,
\[
\bar{m}^f = \sum_{\ell=0}^L m_\ell,
\qquad
\bar{k} = \sum_{\ell=0}^L k_\ell.
\]
The mean $\bar{m}^f$ is the multilevel GP surrogate for $f_L$, and the covariance $\bar{k}$ provides the associated predictive uncertainty.

\subsection{Error analysis}
\label{sec:error_analysis}

We now investigate the multilevel posterior mean error and the posterior variance.
The analysis relies on the following assumptions, which are supposed to hold throughout the rest of this work.

\medskip
\textbf{Assumptions.}
\begin{enumerate}[label=\textbf{(A\arabic*)}, ref=\textbf{(A\arabic*)}]
    \item\label{ass:matern}
        The Matérn kernel $k$ has smoothness $\nu > 0$ such that $\mathcal{H}_k(\Omega) \simeq W_2^\tau(\Omega)$ for $\tau = \nu + d/2$.
    \item\label{ass:solver}
        The numerical solver satisfies the convergence estimate
        \[
            \| f - f_\ell \|_{L^2(\Omega)} \lesssim n_\ell^{-\alpha} \qquad \text{for all } \ell \in \N_0,
        \]
        for some $\alpha > 0$.
    \item\label{ass:increment}
        The level increment $\Delta_\ell = f_\ell - f_{\ell - 1}$ lies in $W_2^{\tau_\Delta}(\Omega)$ for some $\tau_\Delta > d/2$, and
        \[
            \| \Delta_\ell \|_{W_2^{\tau_\Delta}(\Omega)} \lesssim n_\ell^{-\beta} \qquad \text{for all } \ell \in \N_0,
        \]
        for some $0 < \beta \leq \alpha$.
    \item\label{ass:cost}
        The cost of evaluating $\Delta_\ell(\omega)$ at a single point satisfies
        \[
            \mathrm{Cost}(\Delta_\ell) \lesssim n_\ell^{\gamma} \qquad \text{for all } \ell \in \N_0,
        \]
        for some $\gamma > 0$.
\end{enumerate}
\medskip

Assumption \ref{ass:matern} links the kernel smoothness to Sobolev regularity and is standard in the analysis of kernel methods \cite{Wendland2004}, while \ref{ass:solver} quantifies the bias incurred when replacing $f$ by a level-$\ell$ model.
Assumption \ref{ass:increment} characterizes the regularity and decay of the level increments.
The parameter $\beta$ governs the rate at which the increment norms decay with increasing discretization.
The restriction $\beta \leq \alpha$ is no loss of generality, since by \ref{ass:increment} and the telescoping sum \eqref{eq:telescop}, we have
\[
\|f - f_\ell\|_{L^2(\Omega)} \leq \sum_{k > \ell} \|\Delta_k\|_{W_2^{\tau_\Delta}(\Omega)} \lesssim n_\ell^{-\beta},
\]
such that \ref{ass:solver} also holds with rate $\beta$.
Thus we have $\beta \leq \alpha$ whenever $\alpha$ denotes the sharp rate of the solver.
Assumption \ref{ass:cost} models the increasing cost of higher-resolution simulations and is standard in multilevel analyses \cite{Giles2008}.

In the following, we consider a sequence of designs $(D_N)_{N \in \N} \subset \Omega$.
Since Theorem \ref{thm:gp_error} requires a sufficiently small fill distance $h_{D_N} \leq h_0$, we set
\[
N_{\min} \coloneqq \min \{ N \in \N : h_{D_M} \leq h_0 \text{ for all } M \geq N \},
\]
which, using quasi-uniform designs, satisfies $N_{\min} \asymp h_0^{-d}$.

\begin{proposition}[Multilevel GP error bounds]
\label{prop:mlgp_error}
With quasi-uniform designs $D_{N_\ell}$ of size $N_\ell \geq N_{\min}$ on each level and marginal variances estimated by maximum likelihood, the multilevel GP posterior mean $\bar{m}^f$ and covariance
$\bar{k}$ satisfy
\begin{align}
\| f - \bar{m}^f \|_{L^2(\Omega)}
&\lesssim
n_L^{-\alpha} + \sum_{\ell=0}^L N_\ell^{-\min(\tau,\,\tau_\Delta)/d}\, n_\ell^{-\beta},
\label{eq:mlgp_mean_bound}\\[4pt]
\| \bar{k}^{1/2} \|_{L^2(\Omega)}
&\lesssim
\sum_{\ell=0}^L N_\ell^{-\min(\tau,\,\tau_\Delta)/d}\, n_\ell^{-\beta}.
\label{eq:mlgp_variance_bound}
\end{align}
\end{proposition}

\begin{proof}
Throughout, we write $h_\ell$, $q_\ell$, and $\rho_\ell$ for $h_{D_{N_\ell}}$, $q_{D_{N_\ell}}$, and $\rho_{D_{N_\ell}}$, respectively, and use that quasi-uniform designs satisfy $h_\ell \asymp q_\ell \asymp N_\ell^{-1/d}$ and $\rho_\ell \lesssim 1$.

\emph{Posterior mean.} Inserting the finest discretization $f_L$ and applying
the triangle inequality yields
\[
\| f - \bar{m}^f \|_{L^2(\Omega)}
\leq
\| f - f_L \|_{L^2(\Omega)} + \| f_L - \bar{m}^f \|_{L^2(\Omega)}.
\]
By Assumption \ref{ass:solver}, the discretization error satisfies
$\| f - f_L \|_{L^2(\Omega)} \lesssim n_L^{-\alpha}$. For the GP approximation error $\| f_L - \bar{m}^f \|_{L^2(\Omega)}$, the telescoping structure of $f_L = \sum_{\ell=0}^L \Delta_\ell$ and
$\bar{m}^f = \sum_{\ell=0}^L m_\ell$, together with the triangle inequality, and the mean bound (i) of Theorem \ref{thm:gp_error} applied to each increment
$\Delta_\ell \in W_2^{\tau_\Delta}(\Omega)$ gives
\[
\| f_L - \bar{m}^f \|_{L^2(\Omega)}
\leq \sum_{\ell=0}^L \|\Delta_\ell - m_\ell\|_{L^2(\Omega)}
\lesssim \sum_{\ell=0}^L
h_\ell^{\min(\tau,\,\tau_\Delta)}\, \rho_\ell^{\pos{\tau - \tau_\Delta}}\,
\|\Delta_\ell\|_{W_2^{\tau_\Delta}(\Omega)}.
\]
Using quasi-uniformity and Assumption \ref{ass:increment}, that is $\|\Delta_\ell\|_{W_2^{\tau_\Delta}(\Omega)} \lesssim n_\ell^{-\beta}$, yields
\[
\| f_L - \bar{m}^f \|_{L^2(\Omega)}
\lesssim \sum_{\ell=0}^L N_\ell^{-\min(\tau,\,\tau_\Delta)/d}\, n_\ell^{-\beta}.
\]

\emph{Posterior variance.} Since the levelwise GP priors are independent, the
multilevel posterior covariance is additive, that is $\bar{k} = \sum_{\ell=0}^L \hat\sigma_\ell^2\, k_{N_\ell}$, where
$k_{N_\ell}$ is the posterior kernel of the unit-variance Matérn
kernel at level $\ell$. By the triangle inequality,
\[
\| \bar{k}^{1/2} \|_{L^2(\Omega)}
\leq \sum_{\ell=0}^L \|\hat\sigma_\ell\, k_{N_\ell}^{1/2}\|_{L^2(\Omega)}.
\]
Applying Corollary \ref{cor:mle_scaled_var} levelwise to each increment $\Delta_\ell$, together with Assumption \ref{ass:increment}, gives
\[
\|\hat\sigma_\ell\, k_{N_\ell}^{1/2}\|_{L^2(\Omega)}
\lesssim N_\ell^{-\min(\tau,\tau_\Delta)/d}\, \|\Delta_\ell\|_{W_2^{\tau_\Delta}(\Omega)}
\lesssim N_\ell^{-\min(\tau,\tau_\Delta)/d}\, n_\ell^{-\beta}.
\]
\end{proof}

The two bounds in Proposition \ref{prop:mlgp_error} can also be combined into a single error measure that is natural for a probabilistic surrogate.
Since $f_L^{\mathrm{GP}}(\omega) \sim \mathcal{N}\bigl(\bar m^f(\omega),
\bar k(\omega,\omega)\bigr)$ for every $\omega \in \Omega$, the mean squared error of the multilevel surrogate decomposes into bias and variance,
\begin{equation}
\label{eq:mse}
\E \Bigl[\bigl\|f - f_L^{\mathrm{GP}}
\bigr\|_{L^2(\Omega)}^2\Bigr]
= \bigl\|f - \bar m^f\bigr\|_{L^2(\Omega)}^2
+ \bigl\|\bar k^{1/2}\bigr\|_{L^2(\Omega)}^2.
\end{equation}
This error measure accounts for the accuracy of the posterior mean and for the uncertainty simultaneously.

\begin{corollary}[Multilevel root mean squared error]
\label{cor:mlgp_mse}
Under the assumptions of Proposition \ref{prop:mlgp_error}, the mean squared error
\eqref{eq:mse} satisfies
\begin{equation}
\label{eq:mse_bound}
\E \Bigl[\bigl\|f - f_L^{\mathrm{GP}}
\bigr\|_{L^2(\Omega)}^2\Bigr]^{1/2} \;\lesssim\;
n_L^{-\alpha} + \sum_{\ell=0}^L N_\ell^{-\min(\tau, \tau_\Delta)/d}\, n_\ell^{-\beta}.
\end{equation}
\end{corollary}
 
\begin{proof}
Since $\sqrt{a^2 + b^2} \leq a + b$ for $a, b \geq 0$, we have
\[
\E \Bigl[\bigl\|f - f_L^{\mathrm{GP}}
\bigr\|_{L^2(\Omega)}^2\Bigr]^{1/2} \leq \bigl\|f - \bar m^f\bigr\|_{L^2(\Omega)}
+ \bigl\|\bar k^{1/2}\bigr\|_{L^2(\Omega)},
\]
and the two terms are bounded by \eqref{eq:mlgp_mean_bound} and \eqref{eq:mlgp_variance_bound}, which share the same rate.
\end{proof}

The bound \eqref{eq:mse_bound} separates the discretization bias $n_L^{-\alpha}$ at the finest level from the GP variance accumulated across levels.
Note that the posterior variance $\bar k$ only quantifies the sampling error $f_L - \bar m^f$, so that it is an error indicator for $f_L$ rather than $f$.
Reaching a prescribed accuracy requires balancing the
finest discretization $n_L$ against the per-level sample sizes
$\{N_\ell\}_{\ell=0}^L$, subject to the available computational budget.
For brevity, we write $\mu \coloneqq \min(\tau, \tau_\Delta)/d$ for the common
convergence rate of the multilevel posterior mean and variance.

\subsection{Complexity}
\label{sec:complexity}

In order to analyze the complexity of the multilevel GP method, we define its computational cost by the costs required for function evaluations across all levels,
\begin{equation}
    \label{eq:ml_costs}
    \mathcal{C}^{\mathrm{ML}} \coloneqq \sum_{\ell=0}^L N_\ell\,\mathrm{Cost}(\Delta_\ell).
\end{equation}

Note that \eqref{eq:ml_costs} counts only the increment evaluations and not the cost of constructing the GPs, that is, of assembling and solving linear systems across the levels.
In the present setting, where each evaluation requires the solution of a PDE, we assume these additional costs to be negligible compared to $\mathrm{Cost}(\Delta_\ell)$.

\begin{theorem}[Multilevel complexity]
\label{thm:complexity}
For every tolerance $0 < \varepsilon < \mathrm{e}^{-1}$ there exist a number of levels $L$ and sample sizes
$\{N_\ell\}_{\ell=0}^L$ such that
\[
\E \Bigl[\bigl\|f - f_L^{\mathrm{GP}}
\bigr\|_{L^2(\Omega)}^2\Bigr]^{1/2} \lesssim \varepsilon,
\]
with total computational cost bounded by
\begin{equation}\label{eq:complexity}
\mathcal{C}^{\mathrm{ML}}_\varepsilon \;\lesssim\;
\begin{cases}
\varepsilon^{-\frac{1}{\mu}}, & \beta > \gamma\mu,\\[4pt]
\varepsilon^{-\frac{1}{\mu}}\,|\log \varepsilon|^{1+\frac{1}{\mu}}, & \beta = \gamma\mu,\\[4pt]
\varepsilon^{-\left(1 - \frac{\beta}{\alpha}\right)\frac{1}{\mu} - \frac{\gamma}{\alpha}}, & \beta < \gamma\mu.
\end{cases}
\end{equation}
\end{theorem}

\begin{proof}
The proof follows the arguments in \cite[Theorem~4.2]{teckentrupMultilevelStochasticCollocation2015}, applied to the multilevel error bound \eqref{eq:mse_bound} with convergence exponent $\mu$.
For completeness, we present the optimal number of levels and samples explicitly.

\emph{Number of levels.}
The finest level is chosen such that the bias satisfies $n_L^{-\alpha} \lesssim \varepsilon$, which gives
\begin{equation}\label{eq:optimal_L}
L = \left\lceil \frac{1}{\alpha} \log_M(\varepsilon^{-1}) \right\rceil.
\end{equation}

\emph{Sample sizes.}
The samples sizes on each level are chosen such that the computational cost \eqref{eq:ml_costs} is minimized, while the GP approximation error is of order $\varepsilon$.
Treating the sample sizes as continuous and using $\mathrm{Cost}(\Delta_\ell) \lesssim n_\ell^{\gamma}$ from Assumption \ref{ass:cost}, this is the minimization problem
\[
\min_{N_0, \ldots, N_L > 0} \sum_{\ell=0}^L N_\ell\, n_\ell^{\gamma}
\qquad \text{s. t.} \qquad
\sum_{\ell=0}^L N_\ell^{-\mu} n_\ell^{-\beta} = \varepsilon,
\]
with Lagrangian function
\[
\mathcal{L}(N_0, \ldots, N_L, \lambda) = \sum_{\ell=0}^L N_\ell \, n_\ell^{\gamma} + \lambda \left(\sum_{\ell=0}^L N_\ell^{-\mu} n_\ell^{-\beta} - \varepsilon\right).
\]
Taking the derivatives with respect to the sample sizes gives the conditions
\[
\frac{\partial \mathcal{L}}{\partial N_\ell}
= n_\ell^{\gamma} - \lambda\mu\, N_\ell^{-\mu-1}\, n_\ell^{-\beta} = 0,
\]
so that the optimal sample sizes decay geometrically across the levels, namely
\begin{equation}
\label{eq:optimal_allocation}
N_\ell = c\, n_\ell^{-(\beta+\gamma)/(\mu+1)}
\propto M^{-\ell (\beta + \gamma)/(\mu + 1)},
\qquad \ell \in \{0, \ldots, L\},
\end{equation}
with a constant $c > 0$ determined by the constraint.

\emph{Resulting costs.}
Defining $S_L \coloneqq \sum_{\ell=0}^{L} n_\ell^{(\gamma\mu-\beta)/(\mu+1)}$ and inserting \eqref{eq:optimal_allocation} into the constraint shows that $c^{-\mu} S_L = \varepsilon$, that is $c = \varepsilon^{-1/\mu} S_L^{1/\mu}$, such that the total cost is
\begin{equation}
\label{eq:cost_SL}
\mathcal{C}^{\mathrm{ML}}_\varepsilon = c\, S_L
= \varepsilon^{-1/\mu}\, S_L^{\,1 + 1/\mu}.
\end{equation}
It remains to determine the value of the geometric sum $S_L$.
Depending on the sign of $\gamma\mu - \beta$, its ratio $M^{(\gamma \mu - \beta) / (\mu + 1)}$ is smaller than, equal to, or larger than one, so that we have with the optimal number of levels \eqref{eq:optimal_L},
\[
S_L \lesssim
\begin{cases}
1, & \beta > \gamma\mu,\\[2pt]
L \lesssim |\log\varepsilon|, & \beta = \gamma\mu,\\[2pt]
n_L^{(\gamma\mu-\beta)/(\mu+1)}
\lesssim \varepsilon^{-\frac{\gamma\mu-\beta}{\alpha(\mu+1)}}, & \beta < \gamma\mu .
\end{cases}
\]
Inserting these into \eqref{eq:cost_SL} gives exactly the three cases of \eqref{eq:complexity}.

\emph{Minimal sample sizes.}
Finally, Theorem \ref{thm:gp_error} is only applicable for sufficiently small fill distances $h_{D_{N_\ell}} \leq h_0$, so that we replace \eqref{eq:optimal_allocation} by
\begin{equation}
\label{eq:optimal_allocation_floor}
N_\ell = \max\Bigl\{ N_{\min},\
\bigl\lceil c\, M^{-\ell (\beta + \gamma)/(\mu + 1)} \bigr\rceil \Bigr\},
\qquad \ell \in \{0, \ldots, L\} .
\end{equation}
This, however, adds at most $N_{\min} \sum_{\ell=0}^L \mathrm{Cost}(\Delta_\ell) \lesssim \varepsilon^{-\gamma/\alpha}$ to the cost, which is dominated by \eqref{eq:complexity} in all three cases.
If $\beta \geq \gamma\mu$, we have with $\beta \leq \alpha$ that $\gamma\mu \leq \alpha$ and thus $\gamma/\alpha \leq 1/\mu$, while for
$\beta < \gamma\mu$ the exponent in \eqref{eq:complexity} is not smaller than $\gamma/\alpha$, since again $\beta \leq \alpha$.
\end{proof}

As the proof shows, the three cases of Theorem \ref{thm:complexity} distinguish where the work is spent on in the hierarchy.
With the optimal allocation \eqref{eq:optimal_allocation} and $\mathrm{Cost}(\Delta_\ell) \lesssim M^{\gamma \ell}$, due to $n_\ell \propto M^\ell$ and Assumption \ref{ass:cost}, we see that the cost on level $\ell$ is given by
\[
N_\ell\,\mathrm{Cost}(\Delta_\ell) \lesssim M^{-\ell (\beta - \gamma \mu)/(\mu + 1)}.
\]
The sign of the exponent defines the cases, such that if $\beta > \gamma\mu$, the increments decay faster than the evaluations become expensive.
Therefore the work decreases with the level, and the total cost is dominated by the
coarsest level.
If $\beta < \gamma\mu$, the growth in cost is faster, such that the work accumulates on the finer levels.
In the boundary case $\beta = \gamma\mu$ every level contributes the same amount
of work and the logarithmic factor is essentially the number of levels \eqref{eq:optimal_L}.

\begin{remark}[Single-level complexity]
\label{rem:single_level}
For comparison, let $f_{n,N}^{\mathrm{GP}}$ denote the GP conditioned on $N$
evaluations of a single discretization $f_n$, which we refer to as the
\emph{single-level} construction.
By the telescopic structure \eqref{eq:telescop} and Assumption~\ref{ass:increment}, we have
\[
\|f_{n_L}\|_{W_2^{\tau_\Delta}(\Omega)} \leq \sum_{\ell=0}^{L} \|\Delta_\ell\|_{W_2^{\tau_\Delta}(\Omega)} \lesssim 1
\]
uniformly in $L$.
Hence $f_n \in W_2^{\tau_\Delta}(\Omega)$ with norm bounded uniformly in $n$, so that Theorem \ref{thm:gp_error}(i) and Corollary \ref{cor:mle_scaled_var} apply to $f_n$ with the same exponent $\mu$ and, together with Assumption~\ref{ass:solver}, give
\begin{equation}
\label{eq:sl_error}
\E \Bigl[\bigl\|f - f_{n, N}^{\mathrm{GP}}
\bigr\|_{L^2(\Omega)}^2\Bigr]^{1/2} \;\lesssim\; n^{-\alpha} + N^{-\mu}.
\end{equation}
To reach an error tolerance $\varepsilon$, both terms in \eqref{eq:sl_error} are
required to be of order $\varepsilon$, which determines, up to integer rounding,
\[
n^{-\alpha} = \varepsilon
\quad\Longleftrightarrow\quad
n = \varepsilon^{-1/\alpha},
\quad
\text{and}
\quad
N^{-\mu} = \varepsilon
\quad\Longleftrightarrow\quad
N = \varepsilon^{-1/\mu}.
\]
Together with $\mathrm{Cost}(f_n) \lesssim n^\gamma$ by Assumption \ref{ass:cost}, the single-level complexity is thus given by 
\begin{equation}
\label{eq:sl_complexity}
\mathcal{C}^{\mathrm{SL}}_\varepsilon
\;=\; N \cdot \mathrm{Cost}(f_n)
\;\lesssim\; \varepsilon^{-\frac{1}{\mu} - \frac{\gamma}{\alpha}}.
\end{equation}
Note that \eqref{eq:sl_complexity} is an upper bound, which is only sharp if $f_n$ is not smoother than the increments, i.e., if $\tau_f = \tau_\Delta$ for the regularity $\tau_f$ of $f_n$.
\end{remark}

Roughly speaking, while the single-level complexity is the product of the costs for solving the underlying PDE (rate $\gamma / \alpha$) and the GP regression (rate $1/\mu$), the multilevel complexity reduces this to the maximum of those two, up to a logarithmic factor in the case $\beta = \gamma\mu$.
In the coarse-dominated regime $\beta > \gamma\mu$, the multilevel costs reduce to the exponent $1/\mu$, which is equivalent to performing GP regression with evaluations of the \emph{exact} response surface $f$.
On the other hand, in the fine-dominated $\beta < \gamma\mu$, the multilevel GP estimator is as expensive as a \emph{single} PDE solve at a fine discretization $n_L$, i.e., the regression comes for free.

\section{Application to Bayesian Inverse Problems}
\label{sec:bip}

We apply the multilevel GP framework developed in Section~\ref{sec:mlgp} to the
approximation of the posterior distribution arising in an inverse problem.
The strategy is to replace the potentially expensive forward model
$f : \Omega \subset \R^d \to \R$ with the multilevel GP in order to reduce the cost
of approximating the solution of the inverse problem.
If $f$ is vector-valued with values in $\R^{d_y}$, $d_y > 1$, each component could be approximated by an independent multilevel GP, so that we present the scalar case for simplicity.

Consider the Bayesian inverse problem of inferring $\omega \in \Omega$ from a noisy observation
\[
y = f(\omega) + \eta,
\]
where $\eta \sim \mathcal{N}(0, \sigma_{\mathrm{obs}}^2)$ is Gaussian observational
noise with known variance $\sigma_{\mathrm{obs}}^2 > 0$.
Given a prior probability measure $\mu_0$ on $\Omega$ with Lebesgue density $\pi_0 \in L^\infty(\Omega)$, Bayes' theorem yields the posterior measure $\mu^y$ with Lebesgue density
\begin{equation}
\label{eq:true_posterior}
\pi^y(\omega)
\propto
\exp\!\left(-\frac{\bigl(y - f(\omega)\bigr)^2}{2\sigma_{\mathrm{obs}}^2}\right) \pi_0(\omega),
\end{equation}
where the proportionality constant is determined by normalizing to a probability
density.

\subsection{GP based approximation}
\label{sec:gp_based_post}

A standard approach is to replace $f$ in \eqref{eq:true_posterior} by a GP surrogate, yielding an approximate posterior that can be evaluated cheaply.
Most commonly, the deterministic posterior mean $m^f_N$ is used as a plug-in for $f$.
However, this disregards the uncertainty estimate that the GP provides, treating the surrogate as exact and thus ignoring the error it introduces into the approximate posterior \cite{baiGaussianProcessesBayesian2024}.
This uncertainty quantification is especially important in the multilevel setting, where only a few samples are used for the expensive higher-level increments.
Limited amount of data results in substantial posterior uncertainty for the corresponding levelwise GPs.
An alternative is to use a single sample, or to consider the full distribution of the GP.

Following \cite{stuartPosteriorConsistencyGaussian2017}, we adopt the latter
approach and use the marginalization of the multilevel GP posterior with mean $\bar{m}^f$ and variance $\bar{k}$.
We integrate the likelihood in \eqref{eq:true_posterior} over
the GP posterior distribution, which yields the deterministic approximate posterior measure $\mu^y_L$ with density
\begin{equation}
\label{eq:marginal_posterior}
\pi^y_L(\omega)
\propto
\mathbb{E}_{f_L^{\mathrm{GP}}}\!\left[
\exp\!\left(-\frac{\bigl(y - f_L^{\mathrm{GP}}(\omega)\bigr)^2}
{2\sigma_{\mathrm{obs}}^2}\right)
\right] \pi_0(\omega).
\end{equation}
Since $f_L^{\mathrm{GP}}(\omega) \sim \mathcal{N}\!\bigl(\bar{m}^f(\omega),\,
\bar{k}(\omega,\omega)\bigr)$, the expectation in \eqref{eq:marginal_posterior} is a Gaussian convolution and can be evaluated in closed form.
Indeed, carrying out the marginalization over the Gaussian posterior yields
\begin{equation}
\label{eq:marginal_posterior_closed}
\pi^y_L(\omega)
\propto
\frac{1}{\sqrt{\sigma_{\mathrm{obs}}^2 + \bar{k}(\omega,\omega)}}\,
\exp\!\left(
-\frac{\bigl(y - \bar{m}^f(\omega)\bigr)^2}
{2\bigl(\sigma_{\mathrm{obs}}^2 + \bar{k}(\omega,\omega)\bigr)}
\right) \pi_0(\omega),
\end{equation}
where the predictive variance $\bar{k}(\omega,\omega)$ thus enters by
inflating the observation variance $\sigma_{\mathrm{obs}}^2$ to
$\sigma_{\mathrm{obs}}^2 + \bar{k}(\omega,\omega)$.
The following result quantifies the discrepancy between the true posterior
\eqref{eq:true_posterior} and the marginalized GP-based posterior
\eqref{eq:marginal_posterior} in terms of the \emph{Hellinger distance} \cite{stuartInverseProblemsBayesian2010}.
For two probability measures $\mu, \mu'$ on $\Omega$ with Lebesgue densities $\pi, \pi'$, this is defined by
\[
d_{\mathrm{Hell}}(\mu, \mu')
\coloneqq \left(\frac{1}{2}\int_\Omega
\left(\sqrt{\pi(\omega)} - \sqrt{\pi'(\omega)}\right)^{\!2}
d\omega\right)^{\!1/2}.
\]
For a probability density $\pi$ on $\Omega$, we further define the weighted $L^2$ norm
\[
\|g\|_{L^2_{\pi}(\Omega)}^2 \coloneqq \int_\Omega |g(\omega)|^2 \,\pi(\omega) \,d\omega.
\]
Two choices of $\pi$ are of interest, the prior density $\pi_0$ and the posterior density $\pi^y$.

\begin{theorem}[{\cite[Corollary~4.10]{stuartPosteriorConsistencyGaussian2017}, \cite[Corollary~2]{helinIntroductionGaussianProcess2024}}]
\label{thm:hellinger}
Let $\mu^{y,N}$ be the marginalized posterior \eqref{eq:marginal_posterior} associated with the surrogate $\mathrm{GP}(m_N^f, k_N)$. Then it holds
\begin{align}
d_\mathrm{Hell}\!\left(\mu^y,\, \mu^{y,N}\right)
&\;\lesssim
\bigl\|f - m_N^f\bigr\|_{L^2_{\pi_0}(\Omega)}
+ \bigl\|k_N^{1/2}\bigr\|_{L^2_{\pi_0}(\Omega)},
\label{eq:hellinger_prior}\\[4pt]
d_\mathrm{Hell}\!\left(\mu^y,\, \mu^{y,N}\right)
&\;\lesssim
\bigl\|f - m_N^f\bigr\|_{L^2_{\pi^y}(\Omega)}
+ \bigl\|k_N^{1/2}\bigr\|_{L^2_{\pi^y}(\Omega)}.
\label{eq:hellinger_posterior}
\end{align}
\end{theorem}

\begin{remark}[Prior- versus posterior-weighted error]
\label{rem:weighted_error}
In case $\mu_0$ is uniform on the bounded domain $\Omega$, the norm
$\|\cdot\|_{L^2_{\pi_0}(\Omega)}$ coincides with $\|\cdot\|_{L^2(\Omega)}$ up to a factor depending on $|\Omega|$, so that \eqref{eq:hellinger_prior} is an unweighted bound.
The second bound instead reads
\begin{equation}\label{eq:sqrt_post_weighting}
\bigl\|f - m_N^f\bigr\|_{L^2_{\pi^y}(\Omega)}
= \bigl\|(f - m_N^f)\,\sqrt{\pi^y}\bigr\|_{L^2(\Omega)},
\end{equation}
showing that surrogate errors in regions where $\pi^y$ is very small do not contribute.
The two bounds \eqref{eq:hellinger_prior} and \eqref{eq:hellinger_posterior} thus suggest different notions of a good data design.
The first bound is uniform over $\Omega$ and demands a surrogate that is accurate everywhere, while \eqref{eq:hellinger_posterior} suggests controlling $f - m_N^f$ and the posterior variance only on the support of the posterior.
This is the theoretical explanation for the posterior-oriented criteria of
Section \ref{sec:posterior_criteria}, and the resulting designs in the numerical experiments of Section \ref{sec:numerics} show this empirically.
\end{remark}

The surrogate posterior is hence accurate whenever the GP is both accurate in mean and small in variance, so all of the forward error analysis carries over directly to the inverse problem.
Applying bound \eqref{eq:hellinger_prior} of Theorem \ref{thm:hellinger} to the multilevel $\mathrm{GP}(\bar{m}^f, \bar{k})$ surrogate, together with Proposition \ref{prop:mlgp_error} thus directly controls the accuracy of the multilevel posterior.

\begin{corollary}[Multilevel posterior error]
\label{cor:hellinger_ml}
With quasi-uniform designs $D_{N_\ell}$ of size $N_\ell \geq N_{\min}$ on each level and marginal variances estimated by maximum likelihood, the marginalized
multilevel GP posterior \eqref{eq:marginal_posterior} satisfies
\begin{equation}\label{eq:hellinger_ml}
d_\mathrm{Hell}\!\left(\mu^y,\, \mu^y_L\right)
\;\lesssim\;
n_L^{-\alpha}
+ \sum_{\ell=0}^L N_\ell^{-\mu}\, n_\ell^{-\beta}.
\end{equation}
\end{corollary}

\begin{proof}
By \eqref{eq:hellinger_prior} applied to the multilevel surrogate, and since $\pi_0 \in L^\infty(\Omega)$,
\[
d_\mathrm{Hell}\!\left(\mu^y,\, \mu^y_L\right)
\lesssim
\|f - \bar{m}^f\|_{L^2(\Omega)} + \|\bar{k}^{1/2}\|_{L^2(\Omega)}.
\]
The mean estimate \eqref{eq:mlgp_mean_bound} and the variance estimate \eqref{eq:mlgp_variance_bound} of Proposition \ref{prop:mlgp_error} share the same rate, so bounding the two terms yields \eqref{eq:hellinger_ml}.
\end{proof}

Note that the constant in \eqref{eq:hellinger_ml} depends on the noise variance $\sigma_{\mathrm{obs}}$ and typically grows as $\sigma_{\mathrm{obs}}$ decreases, since a small observational noise makes the posterior more sensitive to errors of the
surrogate.
Since the bound \eqref{eq:hellinger_ml} is the same as the forward error estimate \eqref{eq:mse_bound}, the complexity analysis of Section \ref{sec:complexity} directly carries over to the inverse problem.
 
\begin{corollary}[Multilevel complexity for the inverse problem]
\label{cor:complexity_bip}
For every tolerance $0 < \varepsilon < \mathrm{e}^{-1}$ there exist a number of levels
$L$ and sample sizes $\{N_\ell\}_{\ell=0}^L$ such that
\[
d_\mathrm{Hell}\!\left(\mu^y,\, \mu^y_L\right) \lesssim \varepsilon,
\]
with total computational cost bounded by \eqref{eq:complexity}.
\end{corollary}
 
\begin{proof}
By Corollary \ref{cor:hellinger_ml}, the Hellinger distance obeys the same bound as the mean
squared error in Corollary \ref{cor:mlgp_mse}.
The choice of $L$ and $\{N_\ell\}_{\ell=0}^L$ are the same as in the proof of Theorem \ref{thm:complexity}, yielding $d_\mathrm{Hell}(\mu^y, \mu^y_L) \lesssim \varepsilon$ at the stated cost.
\end{proof}

\section{Adaptive multilevel construction}
\label{sec:adaptive}

The a priori, rate-driven construction of Section \ref{sec:mlgp} requires the problem parameters $\alpha$, $\beta$, $\gamma$ and the regularity $\tau_\Delta$ of the increments $\Delta_\ell \in W_2^{\tau_\Delta}(\Omega)$ in order to compute the optimal number of levels and sample sizes from Theorem \ref{thm:complexity}.
In practice, these quantities are rarely known.
In this section, we present an adaptive version of the multilevel GP, which aims to achieve similar accuracy without access to these rates.
This construction is a heuristic not covered by our theory, and its performance is investigated numerically in Section \ref{sec:numerics}.

One first approach could be to estimate the rates from data and then determine the sample sizes \eqref{eq:optimal_allocation} accordingly \cite{Giles2008, Cliffe2011, Collier2015}.
An alternative is to bypass the rates entirely and instead refine the construction incrementally, guided by an error-to-cost ratio updated as samples are added.
Such adaptive refinement strategies are well established in related settings, for instance in the adaptive construction of sparse-grid index sets \cite{Gerstner2003} and in adaptive multilevel quasi-Monte Carlo methods \cite{Giles2009}.

In the following, we adopt the latter approach, by building the multilevel GP incrementally, and placing samples one at a time, while adding levels as needed.
The algorithm refines a current design $D_\ell \subset \Omega$ and the associated level-$\ell$ GP posterior on each level $\ell \in \{0,\ldots,L\}$.
Each refinement step selects a single
\[
(\omega, \ell) \in \Omega \times \{0, \ldots, L\}
\]
pair, evaluates the corresponding increment function $(\omega, \ell) \mapsto \Delta_\ell(\omega)$,
and updates the level-$\ell$ posterior efficiently using a rank-one update.
The choice of $(\ell, \omega)$ is governed by a \emph{sequential design} criterion that trades off the information gained against the cost of the evaluation.
Formally, such a criterion assigns to each level $\ell$ and candidate point $\omega \in \Omega$ an acquisition value $a_\ell(\omega)$, and the next sample is chosen as its maximizer,
\begin{equation}
\label{eq:acquisition}
(\omega^\star, \ell^\star) =
\argmax_{\substack{\omega \in \Omega \\ \ell \in \{0,\ldots,L\}}}\, a_\ell(\omega).
\end{equation}
A natural way to balance information against cost is the ratio
\begin{equation}
\label{eq:gain_cost}
a_\ell(\omega) = \frac{g_\ell(\omega)}{C_\ell},
\end{equation}
with $g_\ell(\omega)$ quantifying the information gain by evaluating the increment function $(\omega, \ell) \mapsto \Delta_\ell(\omega)$, and $C_\ell$ the evaluation cost on level $\ell$.
Different choices of the gain $g_\ell$ yield different design sets.
We distinguish between \emph{forward-oriented} criteria, which depend solely on the surrogate of the forward model $f$ and improve its accuracy uniformly over $\Omega$, and \emph{posterior-oriented} criteria, which additionally use the observation $y$ and target the accuracy of the resulting posterior. Two examples are presented in Sections \ref{sec:forward_criteria} and \ref{sec:posterior_criteria}, respectively.
The ratio \eqref{eq:gain_cost} can also be seen as a \emph{profit} in knapsack-type constructions, where the total gain is maximized subject to a budget on the total work.
This is the standard greedy strategy used in such problems and can be found, for instance, in the construction of sparse-grid index sets \cite{Gerstner2003} or multi-index stochastic collocation \cite{Haji-Ali2016}.

The algorithm starts with level $\ell = 0$ with a small quasi-uniform design $\mathring{D} \subset \Omega$ of size $\mathring{N}$.
At each step, the acquisition criterion selects the level and point offering the largest information gain per unit cost, across all currently active levels $\{0, \ldots, L\}$.
A new level $L + 1$ is added once the criterion has selected the current finest level $\ell^\star = L$.
The per-level costs $C_\ell$ entering the criterion are not known a priori and are estimated online from the measured evaluation times.
The procedure is summarized in Algorithm \ref{alg:adaptive_mlgp}.

\begin{algorithm}
\caption{Adaptive multilevel Gaussian process}
\label{alg:adaptive_mlgp}
\begin{algorithmic}[1]
\REQUIRE response functions $\left\{f_\ell\right\}_{\ell \in \N}$, number of refinement steps $T$; initial design size $\mathring{N}$; acquisition criterion $a_\ell$
\STATE $L \gets 0$; draw a quasi-uniform initial design $\mathring{D}$ of size $\mathring{N}$ and set $D_0 \gets \mathring{D}$
\STATE Evaluate $\Delta_0(D_0)$, fit the level-$0$ GP, and update the cost estimate $C_0$
\FOR{$t = 1, \ldots, T$}
  \STATE Solve for  $(\omega^\star, \ell^\star)$ by the acquisition criterion \eqref{eq:acquisition}
  \STATE Evaluate the increment $\Delta_{\ell^\star}(\omega^\star)$ and set $D_{\ell^\star} \gets D_{\ell^\star} \cup \{\omega^\star\}$
  \STATE Update the level-$\ell^\star$ GP posterior and the cost estimate $C_{\ell^\star}$
  \IF{$\ell^\star = L$}
    \STATE $L \gets L + 1$; draw a quasi-uniform initial design $\mathring{D}$ of size $\mathring{N}$ and set $D_L \gets \mathring{D}$
    \STATE Evaluate $\Delta_L(D_L)$, fit the level-$L$ GP, and update the cost estimate $C_L$
  \ENDIF
\ENDFOR
\STATE \textbf{return} Multilevel $\mathrm{GP}\!\left(\bar{m}^f,\, \bar{k}\right)$ with $\bar{m}^f = \sum_{\ell=0}^L m_\ell$, $\;\bar{k} = \sum_{\ell=0}^L k_\ell$
\end{algorithmic}
\end{algorithm}

\subsection{Forward-oriented criteria}
\label{sec:forward_criteria}

A natural criterion in the GP setting is to place the next point where the posterior is most uncertain.
The \emph{active learning MacKay} (ALM) criterion \cite{MacKay1992} selects the point of maximal posterior variance \eqref{eq:posterior_cov},
\[
\omega_{N+1} = \argmax_{\omega \in \Omega}\, k_{N}(\omega,\omega),
\]
which corresponds to the gain function $g(\omega) = k_N(\omega,\omega)$.
Note that if we were to use a single-level GP, the marginal variance scaling of the kernel does not affect the optimization as a linear factor.
In the multilevel setting, however, $g_\ell(\omega) = k_\ell(\omega, \omega)$ with the MLE-scaled posterior variance $k_\ell = \hat\sigma_\ell^2\, k_{N_\ell}$, and these scales differ between levels, affecting the optimization over the levels.
Inserting the gain function $g_\ell$ into the gain-to-cost ratio \eqref{eq:gain_cost} yields the multilevel ALM acquisition criterion
\begin{equation}\label{eq:alm_cost}
(\omega^\star, \ell^\star) =
\argmax_{\substack{\omega \in \Omega \\ \ell \in \{0,\ldots,L\}}}\,
\frac{k_\ell(\omega,\omega)}{C_\ell}.
\end{equation}
Although the error analysis in Section \ref{sec:error_analysis} is formulated in terms of the integrated variance, the ALM criterion targets its supremum, which controls the integrated variance in the proof of Theorem \ref{thm:gp_error}(ii).
An alternative criterion that aims directly at reducing the integrated variance is the \emph{active learning Cohn} (ALC) criterion \cite{Cohn1996}, at the additional cost of an integral over $\Omega$ per candidate.
Across the levels, however, \eqref{eq:alm_cost} ranks candidates by variance per unit cost, whereas \eqref{eq:mlgp_variance_bound} is linear in $\hat\sigma_\ell$, so that ALM shifts samples towards the coarse levels compared with an allocation minimizing \eqref{eq:mlgp_variance_bound}.
Nevertheless, the numerical experiment in Section \ref{sec:numerics} shows that the resulting complexity is consistent with that of the non-adaptive method.
The ALM strategy is known to produce quasi-uniform designs in the limit \cite{Wenzel2021}, which is also verified in the numerical experiment.
The maximization over $\omega$ is carried out approximately over a fixed quasi-Monte Carlo candidate set.

\subsection{Posterior-oriented criteria}
\label{sec:posterior_criteria}

Generally speaking, forward-oriented criteria reduce the uncertainty of the GP surrogate uniformly over $\Omega$.
However, if the GP is constructed with the goal of approximating the solution of a Bayesian inverse problem, only its accuracy in regions with high posterior mass affects the posterior distribution, which is confirmed by \eqref{eq:hellinger_posterior} in Theorem \ref{thm:hellinger}.
Hence posterior-oriented criteria target the posterior directly \cite{sinsbeckSequentialDesignComputer2017, lartaudSolvingBayesianInverse2026}.

In \cite{sinsbeckSequentialDesignComputer2017}, for instance, the acquisition
criterion is the so-called Bayes risk after a hypothetical evaluation, which we refer to as the \emph{Sinsbeck-Nowak} (SN) criterion.
For a candidate point $\omega \in \Omega$, denote by
\[
\hat Y \sim \mathcal{N}\left(m^f_N(\omega), \,k_N(\omega, \omega)\right)
\]
the $\mathrm{GP}(m^f_N, k_N)$-predicted evaluation at $\omega$, and let $f^{\mathrm{GP}}_+$ denote the GP, that is additionally conditioned on $f(\omega) = \hat Y$.
The gain is then the expected reduction of the Bayes risk,
\begin{equation}
\label{eq:sn_criterion}
g(\omega)
:= \mathbb{E}_{\mu_0}\, \operatorname{Var}_{f^{\mathrm{GP}}}\!\left[ L(\,\cdot \mid f^{\mathrm{GP}}) \right]
- \mathbb{E}_{\mu_0}\, \mathbb{E}_{\hat Y}\, \operatorname{Var}_{f_+^{\mathrm{GP}}}\!\left[ L(\,\cdot \mid f_+^{\mathrm{GP}}) \right],
\end{equation}
with $L$ being the likelihood \eqref{eq:true_posterior}, where the model $f$ is replaced by the random GP surrogates $f^{\mathrm{GP}}$ and $f_+^{\mathrm{GP}}$, respectively.
Note that the first term in \eqref{eq:sn_criterion} is independent of $\omega$, so that it needs to be computed only once per refinement step.
For any estimator of the likelihood, the expected squared $L^2$ deviation from the true likelihood is minimized by the mean $\mathbb{E}_{f^{\mathrm{GP}}}\!\left[L(\,\cdot \mid f^{\mathrm{GP}})\right]$ of the surrogate likelihood, which is precisely the marginalized likelihood underlying the posterior approximation in \eqref{eq:marginal_posterior}.
The Bayes risk attained for this likelihood estimator is the variance appearing in \eqref{eq:sn_criterion}.
Maximizing $g$ therefore chooses the design point such that, when additionally conditioned on it, the resulting marginalized posterior is closest to the true likelihood in this $L^2$ sense.
On the other hand, the plug-in approach, which inserts only the posterior mean $m^f_N$ into the likelihood, ignores the GP uncertainty and is overconfident where the surrogate is inaccurate \cite{sinsbeckSequentialDesignComputer2017}.
The variance as well as the expectation over $\hat Y$ admit closed-form Gaussian expressions \cite{sinsbeckSequentialDesignComputer2017}, and the outer expectation $\mathbb{E}_{\mu_0}$ can be approximated using a fixed quasi-Monte Carlo grid.

For the multilevel setting, candidates consist of pairs $(\ell, \omega)$, and its evaluation conditions only the level-$\ell$ GP on the hypothetical increment value
\[
\hat Y_\ell \sim \mathcal{N}\!\left(m_\ell(\omega),\, k_\ell(\omega,\omega)\right).
\]
Denoting by $f_{L,+}^{\mathrm{GP}}$ the multilevel GP with the level-$\ell$ GP being additionally conditioned on $\Delta_\ell(\omega) = \hat Y_\ell$, the gain becomes
\begin{equation}
\label{eq:sn_gain_ml}
g_\ell(\omega)
:= \mathbb{E}_{\mu_0}\, \operatorname{Var}_{f_L^{\mathrm{GP}}}\!\left[ L(\,\cdot \mid f_L^{\mathrm{GP}}) \right]
- \mathbb{E}_{\mu_0}\, \mathbb{E}_{\hat Y_\ell}\, \operatorname{Var}_{f_{L,+}^{\mathrm{GP}}}\!\left[ L(\,\cdot \mid f_{L,+}^{\mathrm{GP}}) \right].
\end{equation}
Inserting $g_\ell$ into the gain-to-cost ratio \eqref{eq:gain_cost} normalizes by the level cost $C_\ell$, so that design point and level are again chosen simultaneously.
In contrast to the forward ALM criterion, the resulting designs concentrate in regions of high posterior mass and are therefore not quasi-uniform.


\section{Numerical experiment}
\label{sec:numerics}

In this section the derived theoretical rates are illustrated on a Bayesian inverse problem.
We consider the one-dimensional elliptic boundary-value problem
\begin{equation}
\label{eq:elliptic}
-\bigl(a(x,\omega)\, u'(x,\omega)\bigr)' = 1,
\quad x \in U = [0,1],
\qquad
u(0,\omega) = u(1,\omega) = 0,
\end{equation}
with diffusion coefficient $a(x,\omega) = 1 + x + \|\omega\|_2^3$ and uncertain parameter $\omega \in \Omega = [0,1]^2$.
Writing $c_\omega := \|\omega\|_2^3$, the exact solution is
\[
u(x,\omega)
= \log\left(\frac{2+c_\omega}{1+c_\omega}\right)^{-1} \log\left(\frac{1+x+c_\omega}{1+c_\omega}\right) - x.
\]
As the quantity of interest we take the integral of the solution,
\[
f(\omega) = \int_U u(x,\omega)\, dx,
\]
which defines the response surface $f : \Omega \to \R$ to be approximated, and is also available in closed form, namely
\begin{equation}
\label{eq:qoi}
f(\omega) = \frac{3}{2} + c_\omega - \log \left(\frac{2+c_\omega}{1+c_\omega}\right)^{-1}.
\end{equation}
Although this closed form exists, we treat $f$ as accessible only through a numerical solver and use \eqref{eq:qoi} solely to compute reference quantities and corresponding errors.

For the level hierarchy, we define the level-$\ell$ approximation $f_\ell : \Omega \to \R$ by discretizing \eqref{eq:elliptic} on uniform grids with $n_\ell = C_n\, M^\ell$ grid points, using the refinement factor $M = 2$ and $C_n = 10$.
The solver discretizes \eqref{eq:elliptic} using a first-order upwind difference, while the quantity of interest is approximated with the trapezoidal rule.
Due to the first-order upwind discretization we expect Assumptions \ref{ass:solver} and \ref{ass:cost} to hold with $\alpha = 1$ and $\gamma = 1$, which we verify numerically in Section \ref{sec:num_rates}.

The goal then is to infer the true parameter $\omega^\star = (0.3, 0.7) \in \Omega$ from a noisy observation of the quantity of interest,
\[
y = f(\omega^\star) + \eta,
\qquad \eta \sim \mathcal{N}(0, \sigma_{\mathrm{obs}}^2),
\]
with noise level $\sigma_{\mathrm{obs}} = 10^{-3}$.
We use a uniform prior $\mu_0$ on $\Omega$, and the true posterior $\mu^y$ and the GP-based approximations are given by \eqref{eq:true_posterior} and \eqref{eq:marginal_posterior_closed}, respectively.
The error of a surrogate posterior is measured by the Hellinger distance to the reference posterior obtained from the exact response \eqref{eq:qoi}, where the densities are evaluated on an equidistant grid of size $70 \times 70$ over $\Omega$.

We use Matérn kernels \eqref{eq:matern_kernel} with smoothness $\nu = 7/2$ and length scale $\lambda = 1$ for all GPs, such that $\tau = \nu + d/2 = 9/2$, and the marginal variances are estimated levelwise by the MLE \eqref{eq:mle}.
Since the increments depend on $\omega$ only through $c_\omega = \|\omega\|_2^3$, whose $k$-th order derivatives behave like $\|\omega\|_2^{3-k}$ and are square integrable near the origin only for $k < 3 + d/2 = 4$, we have $\Delta_\ell \in W_2^{\tau_\Delta}(\Omega)$ for every $\tau_\Delta < 4$.
The experiment is therefore in the misspecified regime $\tau_\Delta < \tau$, and the effective convergence rate $\mu = \tau_\Delta/d$ can be arbitrarily close to $2$, so that we use $\mu = 2$ in the predicted rates below.

All non-adaptive designs consist of scrambled Sobol points, which are not
quasi-uniform, since their separation radius is not controlled, but attain a fill
distance of order $N^{-1/d}$ \cite{Niederreiter1988}.
In the present misspecified regime, the bounds of Theorem \ref{thm:gp_error} then contain the mesh ratio and separation radius as penalty, such that, strictly speaking, Corollary \ref{cor:mle_scaled_var} and the results of Sections \ref{sec:mlgp} and \ref{sec:bip} do not apply.
However, the observed rates in this experiment still agree with those for quasi-uniform designs, which suggests that the penalties, in particular the one in the MLE bound of Theorem \ref{thm:gp_error} (iii), are pessimistic for low-discrepancy designs.
All kernel matrices are factorized via Cholesky decomposition without a nugget, which
was feasible for the sample sizes considered.

\subsection{Verification of the rate assumptions}
\label{sec:num_rates}

The discretization error of Assumption \ref{ass:solver} is estimated by Monte Carlo using $200$ uniformly drawn $\omega^{(i)} \in \Omega$ and the exact response \eqref{eq:qoi} as reference.
The order of convergence fits well with the order $\alpha = 1$.
We also assume $\beta = \alpha = 1$ for the weaker rate, which is also consistent with the following numerical results.
The cost rate $\gamma$ of Assumption \ref{ass:cost} is estimated from the measured time per solver evaluation.
Due to the implementation, the fit is only visible from the finer levels onwards (roughly $n_\ell = 2000$).
The order is consistent with the linear complexity $\gamma = 1$ of the used solver.
Both fits are shown in Figure \ref{fig:rate_verification}.

\begin{figure}[h]
\centering
\includegraphics[width=\linewidth]{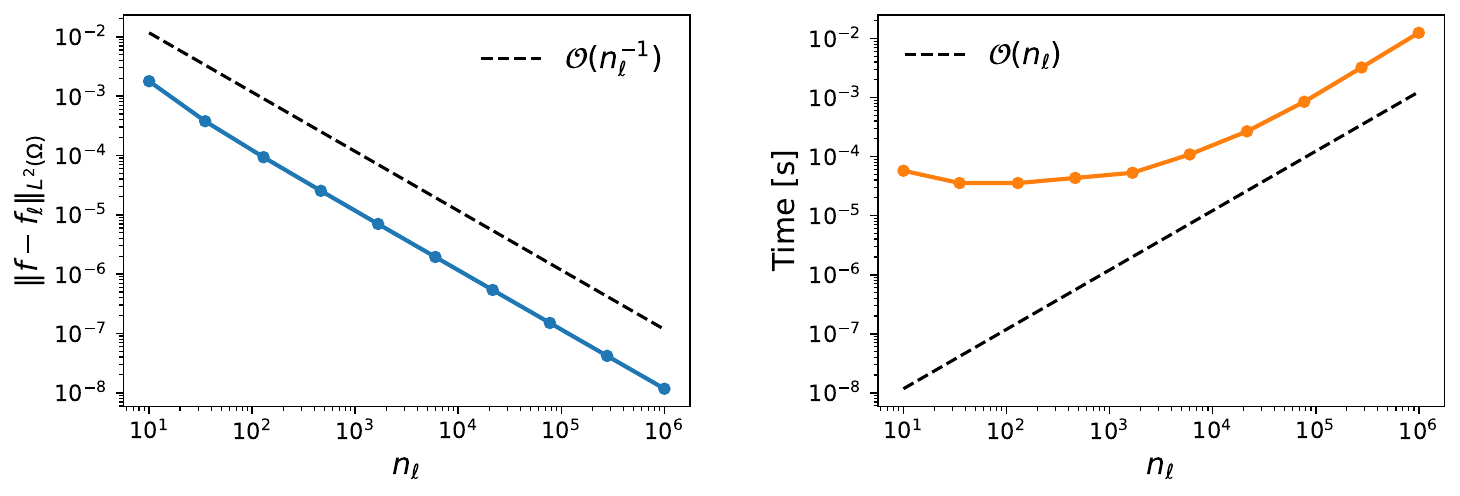}
\caption{Verification of the rate assumptions.
Left: Discretization error $\|f - f_\ell\|_{L^2(\Omega)}$ against the discretization parameter $n_\ell$, with the order $n_\ell^{-1}$ confirming $\alpha = 1$.
Right: Measured time per evaluation, with the order $n_\ell$ consistent with $\gamma = 1$.}
\label{fig:rate_verification}
\end{figure}

\subsection{Multilevel construction}
\label{sec:num_rd}

The non-adaptive, rate-driven multilevel GP fixes the number of levels $L$ and per-level sample sizes $N_\ell$ according to the complexity analysis of Theorem \ref{thm:complexity}.
We run this construction for eight logarithmically spaced target tolerances $\varepsilon \in [10^{-5}, 10^{-2}]$.
Figure \ref{fig:ml_tolerance} (left) shows the observed Hellinger distance $d_{\mathrm{Hell}}(\mu^y, \mu^y_L)$ against the prescribed tolerance $\varepsilon$.
Based on the investigated error decomposition
$$
d_\mathrm{Hell}\!\left(\mu^y,\, \mu^y_L\right)
\lesssim
\|f - f_L\|_{L^2(\Omega)} + \|f_L - \bar{m}^f\|_{L^2(\Omega)} + \|\bar{k}^{1/2}\|_{L^2(\Omega)},
$$
the observed bias, GP approximation error, and posterior standard deviation are shown as well.
All observed errors decay proportionally to $\varepsilon$ over the entire range, in line with the theory for quasi-uniform designs, so that the a priori sample allocation attains the prescribed accuracy up to a moderate constant.

\begin{figure}[h]
\centering
\includegraphics[width=\linewidth]{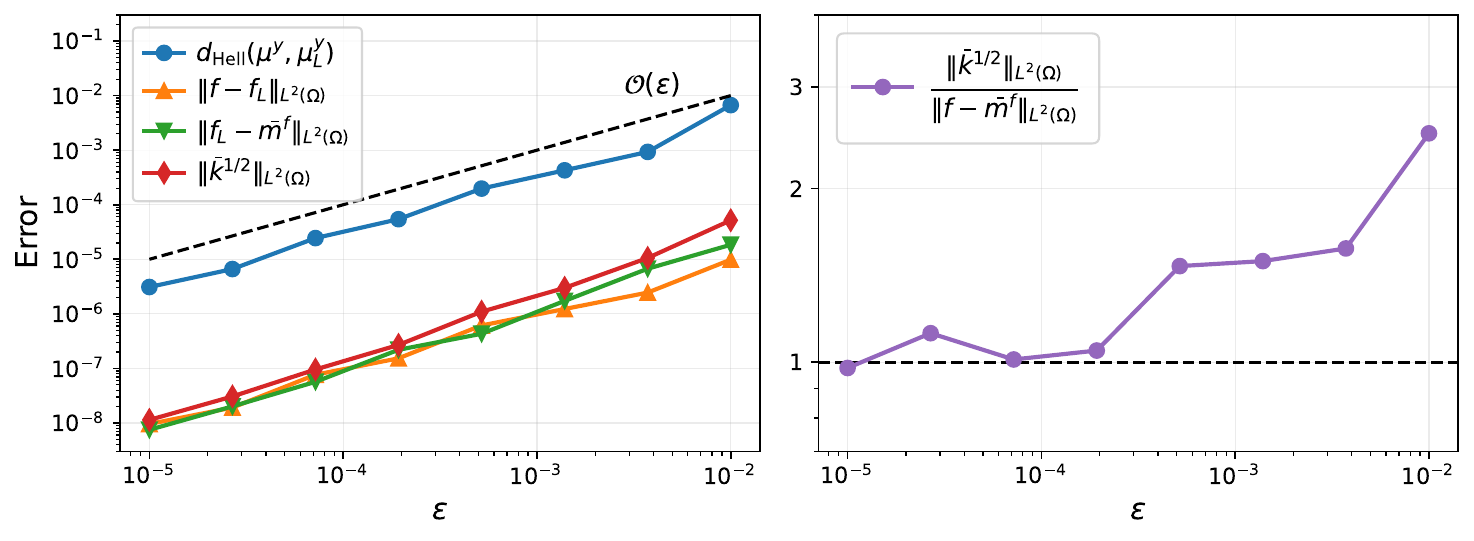}
\caption{Rate-driven multilevel GP against the prescribed target tolerance $\varepsilon$.
Left: Observed Hellinger distance $d_{\mathrm{Hell}}(\mu^y, \mu^y_L)$ (blue), together
with the bias (orange), GP approximation error (green), and integrated posterior
standard deviation (red). The dashed line indicates the order $\varepsilon$.
Right: Ratio of the integrated posterior standard deviation to the error of the
posterior mean. The dashed line indicates the value one.}
\label{fig:ml_tolerance}
\end{figure}

Since the integrated posterior standard deviation $\|\bar k^{1/2}\|_{L^2(\Omega)}$ is available without evaluating $f$, it is natural to ask whether it can be used as an error indicator.
The right plot of Figure \ref{fig:ml_tolerance} shows its ratio with the error $\|f - \bar m^f\|_{L^2(\Omega)}$ of the posterior mean.
Over the considered range of tolerances, the ratio approaches one as $\varepsilon$ decreases, and in this example the posterior standard deviation overestimates the error by at most a factor of $2.5$.
Since both quantities tend to zero and our analysis provides only upper bounds with the same rate for them, it is difficult to predict how their ratio behaves for smaller $\varepsilon$.

\subsection{Complexity comparison}
\label{sec:num_complexity}

We now investigate the predicted computational costs of the multilevel method.
For this, the computational complexity of four surrogate constructions is compared.
The \emph{single-level} GPs (SL) are trained on $N$ Sobol points of one single discretization $f_n$.
The accuracy and cost are thus determined by two parameters, the discretization $n$ and the sample size $N$, which need to be balanced against each other.
To compare the SL method at its most efficient, we evaluate the surrogate on a dense grid of $(n, N)$ combinations and keep only the Pareto-optimal pairs, which means those for which no other combination results in both a smaller error and a smaller computational work.
This yields the balanced pairs
\[
n \in \{8.5 \cdot 10^2,\; 2.7 \cdot 10^3,\; 8.4 \cdot 10^3,\; 2.7 \cdot 10^4,\; 8.4 \cdot 10^4,\; 2.7 \cdot 10^5,\; 8.4 \cdot 10^5\},
\]
with corresponding sample sizes
\[
N \in \{64,\; 96,\; 192,\; 256,\; 512,\; 768,\; 1024\}.
\]
The \emph{rate-driven} multilevel GP (ML) uses the construction of Section \ref{sec:num_rd}.
The two \emph{adaptive} multilevel constructions of Section \ref{sec:adaptive} build the hierarchy incrementally by Algorithm \ref{alg:adaptive_mlgp}.
The forward-oriented variant (ALM) uses the cost-weighted maximum-variance criterion \eqref{eq:alm_cost}, run for $6000$ refinement steps, and the posterior-oriented variant (SN) uses the Bayes-risk criterion \eqref{eq:sn_gain_ml} of Section \ref{sec:posterior_criteria}, run for $1500$ steps.
The outer expectation $\E_{\mu_0}$ of the SN variant is approximated on a quasi-Monte Carlo grid of $512$ points.
For both adaptive variants, the initial Sobol design size for new levels is $\mathring{N} = 5$ and the acquisition criterion is maximized over a fixed Sobol grid of $4096$ points.

\begin{figure}[h]
\centering
\includegraphics[width=\linewidth]{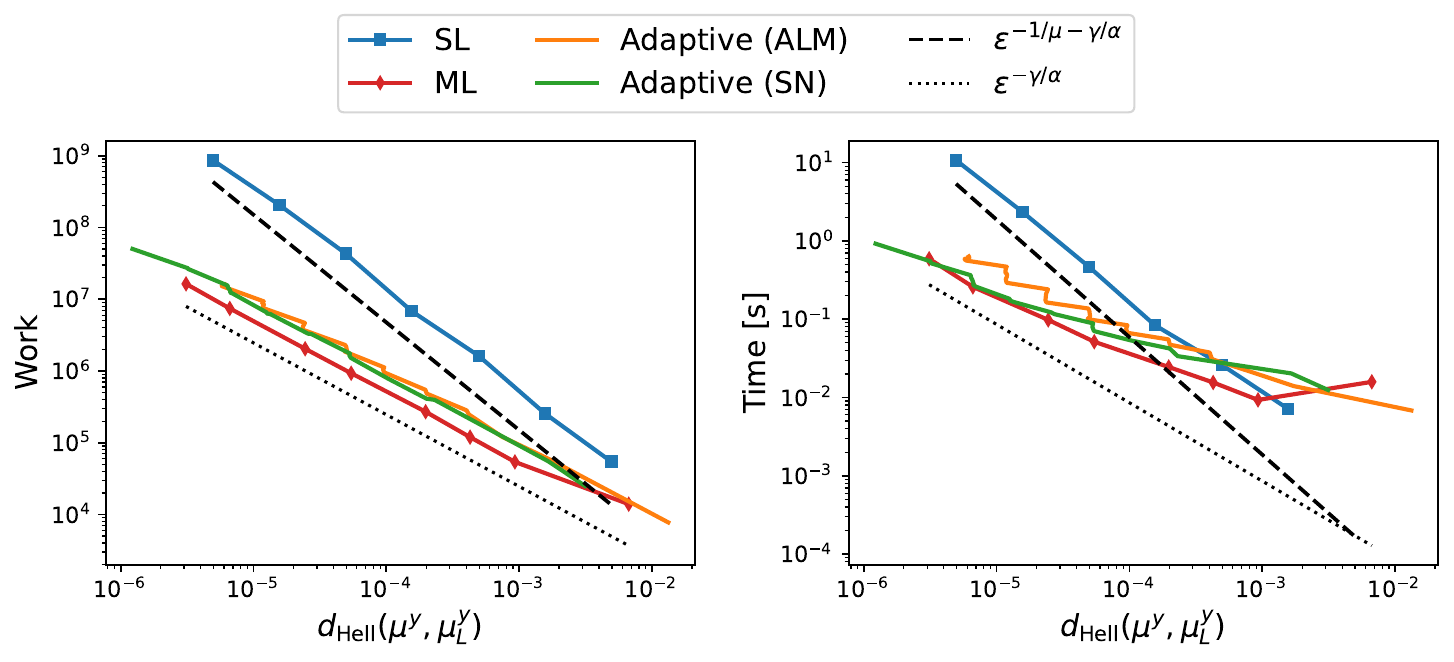}
\caption{Hellinger distance $d_{\mathrm{Hell}}(\mu^y, \mu^y_L)$ to the true posterior as a function of computational work (left) and time (right), for the single-level (SL), rate-driven multilevel (ML), and adaptive multilevel surrogates with forward-oriented (ALM) and posterior-oriented (SN) criteria.
Dashed and dotted lines indicate the theoretical single-level rate $\varepsilon^{-1/\mu - \gamma/\alpha}$ and multilevel rate $\varepsilon^{-\gamma/\alpha}$, respectively.}
\label{fig:complexity}
\end{figure}

Figure \ref{fig:complexity} depicts the Hellinger distance to the true posterior as a function of computational work (left) and time (right).
According to \eqref{eq:ml_costs}, the computational work includes solely the work of response evaluations across all levels, and not the costs for assembling or solving the linear system associated with the kernel matrix.
This is reasonable, since in practice, the costs of evaluating complex computer models dominate these additional costs.
The work is defined by \eqref{eq:ml_costs} with $\mathrm{Cost}(\Delta_\ell) = n_\ell^\gamma$ for the ML constructions, and by $N \cdot n^\gamma$ for the SL surrogates.
The SL complexity curve follows the predicted rate
\[
\mathcal{C}^{\mathrm{SL}}_\varepsilon \lesssim 
\varepsilon^{-1/\mu - \gamma/\alpha} = \varepsilon^{-3/2}.
\]
The rate-driven ML method has the improved rate of Theorem \ref{thm:complexity}.
Since $\beta = 1 < \gamma\mu = 2$, we are in the fine-dominated regime, for which the theorem predicts
\[
\mathcal{C}^{\mathrm{ML}}_\varepsilon \lesssim \varepsilon^{-\left(1 - \frac{\beta}{\alpha}\right)\frac{1}{\mu} - \frac{\gamma}{\alpha}}
= \varepsilon^{-\gamma/\alpha} = \varepsilon^{-1},
\]
where we used $\beta = \alpha$.
Both plots in Figure \ref{fig:complexity} confirm these rates.
The asymptotic cost ratio between the single-level and multilevel methods grows in this case like
\[
\mathcal{C}^{\mathrm{SL}}_\varepsilon / \,\mathcal{C}^{\mathrm{ML}}_\varepsilon = \varepsilon^{-1/\mu} = \varepsilon^{-1/2},
\]
giving savings of around two orders of magnitude at the finest tolerances considered.

The two adaptive constructions attain the same asymptotic rate as the rate-driven method, without any prior knowledge of the problem parameters $\alpha$, $\beta$, $\gamma$.
Comparing the two acquisition criteria, the posterior-oriented SN criterion, which incorporates information about the inverse problem, achieves a posterior accuracy at smaller cost compared to the forward-oriented ALM criterion, i.e., with a smaller constant in the complexity curve.
This is explained by the resulting design sets in Figure \ref{fig:design_set}.

The right plot of Figure \ref{fig:complexity} shows the same comparison in computational time.
Hence this includes, in contrast to the left plot, also the costs for assembling and solving kernel matrix systems at each level, or the time required to evaluate acquisition criteria for the adaptive methods.
We can see that the theoretical rates also predict the resulting computational time well, with a better prediction for smaller tolerances.
The offset in the constant for the ML curves reflects the overhead of fitting a GP on every level and of evaluating the acquisition criterion often for the adaptive methods.
These costs become more negligible compared to the many PDE solves as $\varepsilon$ decreases.

\begin{figure}[h]
\centering
\includegraphics[width=\linewidth]{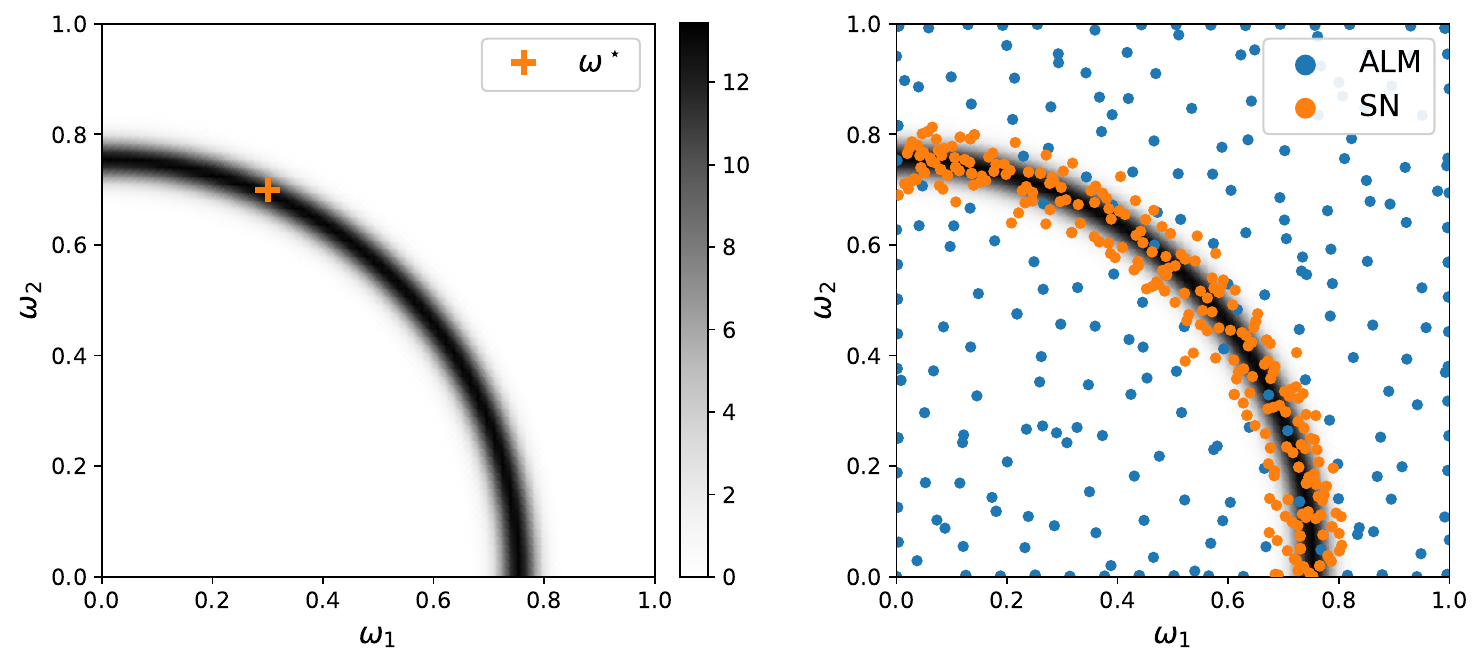}
\caption{True posterior density $\pi^y$ with the true parameter $\omega^\star$ (left), together with the design points selected by the forward-oriented (ALM) and posterior-oriented (SN) adaptive criteria, accumulated over all levels (right).
The ALM design is quasi-uniform over $\Omega$, whereas the SN design concentrates in the region of high posterior mass.}
\label{fig:design_set}
\end{figure}

The right plot of Figure \ref{fig:design_set} compares the design sets produced by the two adaptive acquisition criteria with the true posterior.
Due to the small observational noise $\sigma_{\mathrm{obs}} = 10^{-3}$ and since $f$ depends on $\omega$ only through $\|\omega\|_2$, the posterior concentrates along a thin curved ridge in $\Omega$ passing through $\omega^\star$, shown in the left of the same figure.
The ALM criterion, which has no information of the inverse problem, distributes its design points quasi-uniformly over the whole parameter domain, consistent with the asymptotic quasi-uniformity of maximum-variance designs \cite{Wenzel2021}.
On the other hand, the SN criterion concentrates the majority of its evaluations, across all levels, in the region of significant posterior mass, and only sparsely covers the remainder of the domain.

As mentioned in Section \ref{sec:posterior_criteria}, these designs are not quasi-uniform, so that the results of Sections \ref{sec:mlgp} and \ref{sec:bip} do not apply directly, and the resulting surrogate is accurate only in this region rather than on all of $\Omega$.
Visually, one could argue that it slightly oversamples the tails of the posterior, which is consistent with \cite[Theorem 3]{helinIntroductionGaussianProcess2024}, where the local fill distance should be balanced with the square-root $\sqrt{\pi^y}$ of the posterior density for a good approximation.
This is also reflected in the weighting by $\sqrt{\pi^y}$ in \eqref{eq:sqrt_post_weighting}.
Since surrogate errors outside the support of the posterior do not impact the Hellinger distance, this targeted placement explains the cost advantage of the posterior-oriented criterion observed in Figure \ref{fig:complexity}.

\section{Conclusion}
\label{sec:conclusion}
We have introduced a multilevel Gaussian process (GP) framework for the emulation of functions when a hierarchy of approximations is available.
In case of response surfaces of parametrized PDEs, we modeled the level increments by independent Matérn-kernel GPs and estimated the marginal variance on each level by maximum likelihood, deriving $L^2$ convergence rates for both the posterior mean and the MLE-scaled posterior standard deviation.
These bounds also hold, with a reduced rate, when the kernel is smoother than the function it approximates.
Both bounds have the same rate, so that the mean squared error, which also accounts for the predictive uncertainty, attains the same guaranteed rate as the posterior mean alone.

Combining these bounds with the Hellinger estimate of
\cite{stuartPosteriorConsistencyGaussian2017} yields a complexity result for the multilevel surrogate posterior of a Bayesian inverse problem.
When using only one single GP surrogate, the computational cost is given by the product of the number of samples required by the GP regression and the cost of each PDE solve, whereas the multilevel cost reduces to their maximum.
Depending on whether the increments decay faster or slower than the cost of the solver grows, either the PDE solves or the regression comes for free.

For the practically relevant setting in which the rates $\alpha$, $\beta$, $\gamma$
and the increment regularity are not available a priori, we proposed an adaptive
construction that places samples one at a time and adds levels as needed, guided by
different cost-adjusted acquisition criteria.
In the numerical experiments, both the forward-oriented and the posterior-oriented criterion achieve the asymptotic rate of the non-adaptive multilevel method.
The posterior-oriented criterion concentrates its design in the region of significant posterior mass and reaches a given accuracy at a smaller constant.

The rates degrade as $\min(\tau,\tau_\Delta)/d$, known as the curse of dimensionality, making also the multilevel method difficult to apply in larger parameter dimensions $d$.
Sparse or additive kernels that exploit anisotropic structure can be applied to remedy this \cite{Addy2026}.
Finally, the numerical results suggest that the posterior variance could be used as a computable error, and hence, stopping criterion, which would also allow the adaptive algorithm to target a prescribed tolerance.
This is left as future work.

\section*{Acknowledgments}
This work was supported by KIT MathSEE and the KIT Graduate School for Computational and Data Science through the Bridge PhD grant.
ALT was partially supported by EPSRC grants EP/X01259X/1 and EP/Y028783/1.

\bibliographystyle{siamplain}
\bibliography{references}
\end{document}